\documentclass[11pt]{article}
\usepackage{fullpage}
\usepackage{amsmath,amssymb,amscd}
\usepackage{graphicx}
\usepackage{hyperref}
\usepackage{amsthm,amsfonts}
\usepackage{enumitem}
\newlength{\bibitemsep}
\newlength{\bibparskip}
\let\oldthebibliography\thebibliography
\renewcommand\thebibliography[1]{%
  \oldthebibliography{#1}%
  \setlength{\parskip}{\bibitemsep}%
  \setlength{\itemsep}{\bibparskip}%
}
\usepackage{xcolor}

\theoremstyle{definition}

\newcommand{\lam}{\lambda}

\newtheorem{thm}[equation]{Theorem}

\newtheorem{prop}[equation]{Proposition}

\newtheorem{lem}[equation]{Lemma}

\numberwithin{equation}{section}

\newcommand{\bpm}{\begin{pmatrix}}
\newcommand{\epm}{\end{pmatrix}}
\newcommand{\bsm}{\begin{smallmatrix}}
\newcommand{\esm}{\end{smallmatrix}}
\newcommand{\bspm}{\left(\begin{smallmatrix}}
\newcommand{\espm}{\end{smallmatrix}\right)}
\newcommand{\bm}{\begin{matrix}}
\renewcommand{\em}{\end{matrix}}
\newcommand{\bbm}{\begin{bmatrix}}
\newcommand{\ebm}{\end{bmatrix}}

\newcommand{\bs}{\backslash}

\newcommand{\C}{\mathbb{C}}
\newcommand{\G}{\mathbb{G}}
\newcommand{\A}{\mathbb{A}}
\newcommand{\Q}{\mathbb{Q}}
\newcommand{\Z}{\mathbb{Z}}
\newcommand{\R}{\mathbb{R}}

\newcommand{\Ind}{\operatorname{Ind}}

\newcommand{\la}{\langle}
\newcommand{\ra}{\rangle}

\newcommand{\on}{\operatorname}

\renewcommand{\Re}{\on{Re}}

\newcommand{\gm}{\gamma}

\newcommand{\f}{\mathfrak}

\renewcommand\({\left(}
\renewcommand\){\right)}
\newcommand\srel[2]{\begin{smallmatrix} {#1} \\ {#2} \end{smallmatrix}}
\newcommand{\ttwo}[4]{\(\begin{smallmatrix}{#1} & {#2}
\\ {#3} & {#4} \end{smallmatrix}\)}

\begin{document}

\title{On Arthur's unitarity conjecture for split real groups}

\author{Joseph Hundley\thanks{Supported by National Security Agency grants H98230-15-1-0234 and
H98230-16-1-0125.}~~~({\tt jahundle@buffalo.edu})\\
Stephen D. Miller\thanks{Supported by National Science Foundation grants DMS-1801417 and CNS-1815562.}~~~({\tt{miller@math.rutgers.edu}})}

\maketitle

\vspace{-.6cm}

\begin{abstract}
Arthur's conjectures predict the existence of some very interesting unitary representations occurring in spaces of automorphic forms.  We prove the unitarity of the ``Langlands  element'' (i.e., the representation specified by Arthur) of all unipotent Arthur packets for split real exceptional groups.   The proof  uses Eisenstein series, Langlands' constant term formula and square integrability criterion, analytic properties of intertwining operators, and some mild arithmetic input from the theory of Dirichlet $L$-functions, to reduce to a more combinatorial problem about intertwining operators.  \textcolor{blue}{This updated arXiv posting also includes some comments (in blue) concerning statements about normalized intertwining operators we quoted from the literature in Section~\ref{sec:intertwine}.}
\end{abstract}

\section{Introduction}

In \cite{Arthur-conjectures1,Arthur-conjectures2} James Arthur introduced a series of conjectures about the discrete automorphic spectrum.  Motivated by his work on the trace formula, they predict several properties of   packets of so-called ``unipotent'' representations, all of which are conjectured to occur unitarily in spaces of automorphic forms on the adelic points of some  connected reductive linear algebraic group $G$.  Arthur's conjectures for many classical groups were solved in \cite{Arthur-book} using the deep input of Ngo's proof of the Fundamental Lemma \cite{Ngo}.

Though Arthur's conjectures are global in nature, they have deep local implications to  representation theory.  In the particular setting of real reductive groups, the  unitary representations they predict have often been difficult to understand.  The book \cite{ABV} provides a  detailed study of the properties of Arthur's unipotent representations for the real points $G(\R)$ of $G$, and proves all of their major expected properties -- except for unitarity. The unitarity is most difficult for split $G$, where it has been solved for spherical representations in \cite{M} using Eisenstein series (completing earlier work of \cite{Kim-G2,Moeglin,GMV}).

This paper studies the non-spherical representations of split $G(\R)$ using more complicated Eisenstein series.  Arthur's book \cite{Arthur-book} proves his conjectures  for the classical groups $SO(n,n)$, $SO(n,n+1)$, and split $Sp(2n)$ (but not for  Spin covers or  the half-spin group $HSpin(2n)$), while Henry Kim earlier used Eisenstein series to prove our Theorem~\ref{thm:mainthm} below for $G_2$ (where it also follows from Vogan's classification of the unitary dual \cite{Vogan}).  Though our methods surely apply to classical groups as well, we focus on the remaining exceptional group cases (for which no general pattern exists).

  Arthur's formulation \cite{Arthur-conjectures1,Arthur-conjectures2} of his conjectures does not explicitly define his  packets aside from describing a particularly canonical ``Langlands element'' of each (see (\ref{langlandsparam}) and Theorem~\ref{thm:unitarizablequotient}).  A   description of the  full Arthur packets for real groups was given in  \cite[\S22.6]{ABV} in terms of the Beilinson-Bernstein classification \cite{BB}.  It is this Langlands element which we demonstrate the unitarity of; we do not say anything about  the other elements of the packet, which are difficult to concretely identify.
\begin{thm}\label{thm:mainthm}
Let $G$ be a Chevalley group other than $Spin(n+1,n)$, $Spin(n,n)$, or $HSpin(2n,2n)$ (the ``half-spin'' group double-covered by $Spin(2n,2n)$).  Then the canonical ``Langlands element" of each unipotent Arthur packet of representations of $G(\R)$ is unitarizable.
\end{thm}
\noindent
This will be obtained as a consequence of Theorem~\ref{thm:unitarizablequotient} below.

Though our argument relies on   an analysis of intertwining operators and constant terms of Borel Eisenstein series, it is significantly different than the methods used earlier for the spherical representations in \cite{M}.  This is because here we directly construct the  full representation as a quotient using Eisenstein series, and must   deal with intertwining operators as operator-valued meromorphic functions (in particular, facing subtle issues of analytic continuation).  In contrast, the argument in \cite{M} only required (scalar-valued) intertwining operators on spherical vectors.  In fact, the analysis is sufficiently different  that  specializing the present argument to the spherical case does not recover the method of \cite{M}.

  Our global methods
  show that certain  partial residues of Eisenstein series are square-integrable. Consequently each of their local component representations is unitary, in particular their archimedean components (which realize the representations in Theorem~\ref{thm:mainthm}).  One also simultaneously deduces the unitarity of the nonarchimedean representations attached to these residues, which is itself an interesting aspect of Arthur's conjectures; however, these nonarchimedean representations could already be shown to be unitary using other  methods \cite{BC,BM}.

We use intertwining operators to reduce square-integrability to   a combinatorial calculation in the root system.  The specific technical scheme which accomplishes this in Section~\ref{sec:proofofThm4.3} unfortunately  does not literally apply to four particular representations we would like to study; modifying it to do so is likely possible, but would be lengthy and sacrifice uniformity.  Additionally,
though Arthur predicts
 certain unipotent representations are unitarily induced from proper Levi subgroups
 \cite[pp.~43-44]{Arthur-conjectures2}, proving this requires case-by-case calculations; these representations could also be studied using Eisenstein series, but again at the cost of length and uniformity.  For both of these reasons, we have elected to instead leverage  the recent algorithmic progress on the unitary dual problem \cite{The5} and found it simpler to verify the unitarity of the remaining representations (i.e., those not covered in Section~\ref{sec:proofofThm4.3}) using the {\tt{atlas}} software \cite{atlas}.  In fact, our methods are thankfully complementary in that they work best on the cases that are most difficult for {\tt{atlas}}, and vice-versa.  Thus although we expect Eisenstein series to apply more generally than to what we have used them for in this paper, using the {\tt{atlas}} software greatly streamlines the presentation of the paper.
%
%

We would like to acknowledge Jeffrey Adams, James Arthur, Dan Barbasch, Dan Ciubotaru, Brian Conrad, Howard Garland, James Humphreys, Richard Lyons, Allen Moy, Siddhartha Sahi, Wilfried Schmid, Freydoon Shahidi, Peter Trapa, and David Vogan for their   helpful conversations and advice.  We are particularly indebted to David Vogan for important initial conversations and for suggesting the strategy in Section~\ref{sec:unipotentarthurparmeters} of listing Arthur parameters first by the image of $W_\R$, and only secondly by the image of a commuting $SL(2,\C)$ subgroup.  Likewise, we owe a large debt of gratitude to Jeffrey Adams for numerous lengthy discussions about Arthur packets and the role of non-distinguished orbits, as well as for guidance on using the {\tt atlas} software.

\section{Exceptional Chevalley Groups}\label{sec:chevalleygroups}

In this section we describe the group-theoretic background necessary for
defining the Eisenstein series employed in proving Theorem~\ref{thm:mainthm}.  Like all automorphic forms, Eisenstein series are functions on a Lie group which are invariant under a discrete arithmetic subgroup.  For notational convenience -- and in the spirit of the global origin of Arthur's conjectures -- we will instead formulate them as functions on the adelic points of a  linear algebraic group.  In this section we review some basics of linear algebraic groups and their real and adelic points.

As in the introduction, $G$ is assumed to be a connected reductive linear algebraic group defined and split over $\Q$.  (Our techniques work over other fields of definition, but $\Q$ is adequate to capture all the representations in Theorem~\ref{thm:mainthm}).
  By standard reductions determining the representations of $G(\R)$ in terms of its factors, we may (and do) additionally assume that $G$ is    simple.   Since Arthur \cite{Arthur-book} has proven his conjectures for the classical groups $SO(n,n)$, $SO(n,n+1)$, and split $Sp(2n)$ (but not for  Spin covers or  the half-spin group $HSpin(2n)$),  we focus on exceptional groups and assume for the rest of the paper that $G$ is of exceptional type.   A   list of the seven such $G$ is given in Table~\ref{fig:dualgroups}.

\begin{table}$$\text{
\begin{tabular}{|c|ccccccc|}
\hline
$G$ & $G_2$ & $F_4$ & $E_6^{sc}$ & $E_6^{ad}$ & $E_7^{sc}$ & $E_7^{ad}$ & $E_8$ \\
\hline
$G^\vee$  & $G_2$ &  $F_4$& $E_6^{ad}$ & $E_6^{sc}$ & $E_7^{ad}$  & $E_7^{sc}$  & $E_8$
\\
\hline
\end{tabular}}$$
\caption{Exceptional Chevalley groups $G$ and their duals $G^\vee$.  Recall that there are unique Chevalley groups of type $G_2$, $F_4$, and $E_8$, while there exists both an adjoint form $E_n^{ad}$ and a simply connected form $E_n^{sc}$ for $n=6$ and $7$; in these cases $E_n^{ad}$ is a quotient of $E_n^{sc}$ of order 3 (if $n=6$) or order 2 (if $n=7$).  The arrows in the Dynkin diagram are also reversed in passing from $G$ to $G^\vee$ (see Figure~\ref{fig:Dynkin}).}
\label{fig:dualgroups}
\end{table}

These $\Q$-split linear algebraic groups are sometimes referred to as {\it Chevalley groups}.  However, the usage  of this terminology in the literature has changed somewhat with time.  In order to fit with more modern references, we will instead reserve the   term to mean a split, connected reductive group-scheme over $\Z$ having an irreducible root system.
Actually, such  group-schemes have the same classification as linear algebraic groups having irreducible root systems\footnote{This classification   for group schemes is found in \cite[Th\'eor\`eme 1.1, Expos\'e XXV]{SGA3}, while for linear algebraic groups over algebraically closed fields it is found in
\cite[Theorem 9.6.2, 10.1.1]{Springer}; see
 \cite[16.3.2, 16.3.3]{Springer} for the classification of split groups over any field.};  the group scheme vantage-point thus adds a $\Z$-structure to them.  We shall next review this notion in more concrete terms, as well as explain the connection with Steinberg's presentation \cite{Steinberg} for the simply connected form.  At this point it may be useful to note that
 Theorem~\ref{thm:mainthm} reduces to the case that $G$ is  simply connected because of some properties of $G(\R)$:~$E^{ad}_6(\R)=E^{sc}_6(\R)$, while $E_7^{ad}(\R)$ is a disconnected Lie group whose connected component is the image of the 2-to-1 projection map from $E^{sc}_7(\R)$.\footnote{This is analogous to the   relationship between $SL_n(\R)$ and $PGL_n(\R)$, for $n=3$ and $2$ (respectively).}  In both cases the unitarity of representations of the real points of the adjoint form is determined by those of the simply connected form, e.g., by induction.  Nevertheless, the complex points of the adjoint form $G^\vee$   play a fundamental role in the formulation of Arthur's conjectures (see Section~\ref{sec:unipotentarthurparmeters}).

We shall now review the Chevalley  construction of the five exceptional, simply-connected, split group-schemes $G$ over $\Z$.  For $r\in \{2,4,6,7,8\}$ let   $\Delta\subseteq \R^r$ denote the unique (up to scaling and rotation)  exceptional root system of rank $r$, endowed  with   inner product $(\cdot,\cdot)$.  Let $\Sigma=\{\alpha_1,\ldots,\alpha_r\}$ denote a fixed choice of simple roots and  $\Delta_+$ the corresponding positive roots in $\Delta$.
 Let $\frak g_\C$ denote the unique complex semisimple Lie algebra corresponding to $\Delta$.
  The Chevalley basis of $\frak g_\C$ consists of positive root vectors $X_\alpha$ and  negative root vectors $X_{-\alpha}$  for each $\alpha\in\Delta_+$, and neutral elements $H_\alpha=[X_\alpha,X_{-\alpha}]$ for each  $\alpha\in \Sigma$.  The  $\Z$-span of the Chevalley basis in $\frak g_\C$ is called the Chevalley lattice.

Each Chevalley basis element  defines a one-parameter algebraic subgroup of $G$. For the nilpotent elements  $X_{ \alpha}$, these can be identified as  $u_{\alpha}(s)=\exp(s X_{\alpha})$, $\alpha\in \Delta$.  This identification can be made concrete  by realizing the Chevalley basis as   integral matrices, so that this exponential itself defines an algebraic matrix group isomorphic to ${\mathbb{G}}_a$.
 For  $\alpha\in \Delta$  let
\begin{equation}\label{walphat}
    w_\alpha(s) \ \ = \ \ u_{\alpha}(s) \,  u_{-\alpha}(-s^{-1}) \, u_{\alpha}(s) \ \ \ \ \text{and} \ \ \ \ h_{\alpha}(s) \ \ = \ \ w_{\alpha}(s)\,w_{\alpha}(1)^{-1}\,.
\end{equation}
The map $h_\alpha(\cdot)$ defines a one-parameter subgroup isomorphic to ${\mathbb{G}}_m$ corresponding to $H_\alpha$.
When $z$ is a complex number,  $h_\alpha(e^z)$ can be identified with $\exp(zH_\alpha)\in G(\C)$;~for a positive real number  $s$,
 $h_\alpha(s)$ can be identified with $\exp(\log(s)H_\alpha)\in G(\R)$.   The product of the $r$ ${\mathbb{G}}_m$-subgroups $\{h_\alpha(\cdot)\}$, $\alpha\in\Sigma$, is a split maximal torus $T$ of $G$.
    The roots   of $T$ in $G$ can be identified with $\Delta$, and
 the Weyl group $W$ of $\Delta$ is isomorphic to (and will be identified with) with the quotient of the normalizer of $T$
 by its centralizer.
  By definition the torus $T$ acts on $X_\alpha$ under the adjoint action by the character written as $t\mapsto t^\alpha$, for each $\alpha \in \Delta$ (we will identify $\alpha$ as this character, written using this exponential notation).
Let $\frak a_\C$  denote the complex span  of the  $H_\alpha$, which is a Cartan subalgebra of $\frak g_\C$ and the Lie algebra of $T$.  The $\R$-span $\frak a$ of the $H_\alpha$ is the Lie algebra of the real Lie group $T(\R)$.
Let $N\subset G$ (resp., $N_{-}\subset G$) denote the  maximal unipotent subgroup containing each subgroup $\{u_\alpha(\cdot)\}$ (resp. $\{u_{-\alpha}(\cdot)\}$  ), for all $\alpha\in \Delta_+$.  The Borel subgroup is the semidirect product $B=N\rtimes T$.

We now describe some facts which depend crucially  on the assumption that $G$ is simply connected.
For $F=\Q$ or a completion $\Q_v$  (where $v\le \infty$ is a place of $\Q$ and $\Q_\infty=\R$ by convention), the $F$-points $G(F)$ of $G$ are   generated by $\{u_\alpha(s)|\alpha\in\Delta,s\in F\}$.  In particular, $G(\R)$ is a connected Lie group (in contrast to the situation for  $E_7^{ad}(\R)$ mentioned above).
Elements of $T(\Q_v)$ have unique expressions as   products
\begin{equation}\label{uniquefact}
 h_{\alpha_1}(s_1)h_{\alpha_2}(s_2)\cdots h_{\alpha_r}(s_r)\,, \ \ \text{for some} \ s_1,\ldots,s_r\in \Q_v^\star\,,
\end{equation}
 \cite[Lemma~35]{Steinberg}.  Elements of $N(\Q_v)$ also can be uniquely expressed as products of factors $u_\alpha(x_\alpha)$, with the product taken over $\alpha\in\Delta_+$ in some fixed order and $x_\alpha\in \Q_v$ \cite[Theorem~7]{Steinberg}.  We choose (as we may) a full set of Weyl group representatives amongst the subgroup generated by the elements $\{w_\alpha(1)|\alpha\in\Sigma\}$ from (\ref{walphat}), the latter of which correspond to the simple reflections $w_\alpha$ of the root system (see \cite[Lemma~22]{Steinberg}).  By this and the Bruhat decomposition $G(\Q_v)=\cup_{w\in W}B(\Q_v)wB(\Q_v)$, each element of $G(\Q_v)$ can be   written as a product of elements in the above one-parameter subgroups corresponding to Chevalley basis elements.  Relations satisfied by these elements are given in \cite[\S6]{Steinberg}.

  The integral points $G(\Z)$ are defined as the stabilizer  in $G(\Q)$ of the Chevalley lattice under the adjoint action on $\frak g_\C$.  Alternatively, they
  are generated by the elements $u_\alpha(1)$, $\alpha\in \Delta$ \cite[p.~115]{Steinberg}.  In particular, $G(\Z)$ contains $w_\alpha(1)$ for all $\alpha\in \Delta$, and hence our full set of Weyl group representatives from the previous paragraph.  Likewise, for $p<\infty$ let
   $K_p=G(\Z_p)$   denote the stabilizer in $G(\Q_p)$ of the Chevalley basis tensored with $\Z_p$.  The adele group $G(\A)$ is the restricted direct product of all $G(\Q_v)$, $v\le \infty$, with respect to the $G(\Z_p)$, $p<\infty$.
 For any place $v\le \infty$ of $\Q$, the Iwasawa decomposition asserts that $G(\Q_v)=B(\Q_v)K_v$, where $K_p=G(\Z_p)$ for $p<\infty$ and $K_\infty$ is the maximal compact subgroup of $G(\R)$ generated by the one parameter subgroups $\{\exp(t(X_\alpha-X_{-\alpha}))|t\in \R\}$, $\alpha\in \Sigma$ (see \cite[Proposition 2.33]{Iwahori-Matsumoto}
or \cite[p.~114]{Steinberg}).

Above we identified the root system $\Delta$ with a set of characters of $T$, written as   $t\mapsto t^\alpha$.  This notation extends by linearity to define characters $t\mapsto t^\lambda$ for any element $\lambda$ of the root lattice $\Lambda_{\on{rt}}=\Z\alpha_1\oplus\Z\alpha_2\oplus\cdots \oplus \Z\alpha_r$.  Composition with the exponential map from $\frak a$ to $T(\R)$ allows us to identify these characters  with elements of ${\frak a}^*$, the space of real-valued linear functionals on $\frak a$.  We write the (nondegenerate) pairing  between ${\frak a}^*$ and $\frak a$ as $\langle \cdot,\cdot \rangle:{\frak a}^*\times {\frak a}\rightarrow \R$.   For each root $\alpha\in \Delta$, let $\alpha^\vee=\frac{2}{( \alpha,\alpha)}\alpha$ denote its coroot, regarded as an element of $\frak a$ using the nondegeneracy of the two pairings $(\cdot,\cdot)$ and $\langle \cdot,\cdot\rangle$.

Let
  $\varpi_1,\ldots,\varpi_r \in {\frak a}^*$ denote the fundamental weights  characterized by the inner product relations
\begin{equation}\label{varpidef}
   \langle\varpi_i,\alpha_j^\vee\rangle \ \ = \ \ \delta_{i=j} \ \ =  \ \ \left\{
                                                      \begin{array}{ll}
                                                        1, & i=j\,, \\
                                                        0, & i\neq j\,.
                                                      \end{array}
                                                    \right.
\end{equation}
We let $\Lambda_{\on{wt}}=\Z\varpi_1\oplus\Z\varpi_2\oplus\cdots \oplus \Z\varpi_r$ denote the weight lattice, which contains the root lattice $\Lambda_{\on{rt}}$:
\begin{equation}\label{latticechain}
  \Lambda_{\on{rt}} \ \ \subseteq \ \ \Lambda_{\on{wt}} \ \ \subseteq \ \ {\frak a}^*.
\end{equation}
The Chevalley basis elements satisfy $[H_\alpha,X_\beta]=\langle\beta,\alpha^\vee\rangle X_\beta$, and thus $h_\alpha(s)X_{\beta}h_\alpha(s)^{-1}=s^{\langle \beta,\alpha^\vee\rangle} X_\beta$  for $\alpha$, $\beta \in \Delta$ \cite[p.~68]{Steinberg}; in particular,
\begin{equation}\label{corootrootinteraction}
    h_\alpha(s)^\beta \ \ = \ \ s^{\langle \beta,\alpha^\vee\rangle}\,,
\end{equation}
reflecting that $h_\alpha$ is associated with the coroot $\alpha^\vee$.  This motivates the alternative notation
\begin{equation}\label{coroothsnotation1}
    \alpha^\vee(\cdot ) \ \ = \ \ h_\alpha(\cdot)
\end{equation}
for the algebraic map from $\mathbb G_m$ to $T$ defined in (\ref{walphat}).  More generally, if an element   $\varpi^\vee$ of the coroot lattice $\Z\alpha_1^\vee\oplus \Z\alpha_2^\vee\oplus\cdots \oplus \Z\alpha_r^\vee$
 is written as $\sum_{i\le r}c_i \alpha_i^\vee$, $c_i\in \Z$, we use the notation
\begin{equation}\label{coroothsnotation2}
    \varpi^\vee(s) \ \ = \ \  \alpha_1^\vee(s)^{c_1}\cdots \alpha_r^\vee(s)^{c_r}  \ \ = \ \  h_{\alpha_1}(s)^{c_1}\cdots h_{\alpha_r}(s)^{c_r}
\end{equation}
so that
\begin{equation}\label{coroothsnotation3}
    \varpi^\vee(s)^\beta \ \ = \ \ s^{\langle \beta,\varpi^\vee\rangle}\,, \ \ \ \beta\,\in\,\Delta\,,
\end{equation}
consistent with   (\ref{corootrootinteraction}).

 Suppose $\chi$ is an algebraic character of $T$; then $\chi\circ h_{\alpha_i}$ must itself be an algebraic character of $\mathbb G_m$,  hence an integral power. Therefore $\chi$ is determined by integers $c_1,\ldots, c_r$ such that $\chi(h_{\alpha_i}(s))=s^{c_i}$ for all $i\le r$.  In particular, $\chi$ determines an integral weight  $\lambda=\sum_{i\le r}c_i\varpi_i\in \Lambda_{\on{wt}}$; we shall write $\chi(t)=t^\lambda$, a notation consistent with the earlier usage for  $\lambda\in \Lambda_{\on{rt}}$ as well as (\ref{corootrootinteraction}).  Therefore the full set of algebraic  characters $X(T)$ can be viewed as a subset of the weight lattice,
\begin{equation}\label{XTlattice}
  \Lambda_{\on{rt}} \ \ \subseteq \ \ X(T) \ \ \subseteq \ \ \Lambda_{\on{wt}}\,.
\end{equation}
 Like all characters defined on $T$, elements of $X(T)$ extend trivially on $N$ to algebraic characters of $B=N\rtimes T$; we will also denote such characters as  $b\mapsto b^\lambda$, for some $\lambda\in \Lambda_{\on{wt}}$.
Note  that the containments in (\ref{XTlattice}) may in general be proper.   Since we have been assuming $G$ is simply connected, $X(T)=\Lambda_{\on{wt}}$ and the full set $X^\vee(T)$ of morphisms from $\mathbb G_m$ to $T$ coincides with the coroot lattice.
  All three lattices in (\ref{XTlattice}) coincide for $G_2$, $F_4$, and $E_8$.  The center of $G$ and its character group are both isomorphic to the finite quotient $X(T)/  \Lambda_{\on{rt}}$, which is cyclic of order  3 for $E_6$ and of order 2 for $E_7$; $Z(G)$ is explicitly described in    (\ref{centers}) below.

The Weyl group $W$
 acts on all three lattices in (\ref{XTlattice})  and  their common complexification, denoted ${\frak a}^*_{\C}$.
  Each $w\in W$, when regarded as an element of $G(\Z)$ as above, normalizes $T$ and satisfies the relation
   \begin{equation}\label{wtwvarpi}
     (wtw^{-1})^\varpi  \ \ = \ \  t^{w^{-1}\varpi}
   \end{equation}
    for any $\varpi\in X(T)\subseteq \Lambda_{\on{wt}}$ ($wtw^{-1}$ is independent of the
choice of representative for $w$).
Likewise, there
is a natural action of $w\in W$ on elements $\varpi^\vee\in X^\vee(T)$  such that
\begin{equation}\label{wvarpi}
  (w \varpi^\vee)(\cdot) \ \  = \ \  w\varpi^\vee(\cdot)w^{-1}\,,
\end{equation}
 again independent of the choice of representative of $w$.

We conclude this section with some comments about non-simply connected $G$, whose appearance in this paper is limited to the complex points of  $E_6^{ad}$ and $E_7^{ad}$.  These $G$ are both quotients  of the simply connected algebraic group by its (finite) center:
\begin{equation}\label{centers}
\aligned
  Z(E_6^{sc}) \ \  & = \ \ \{e,z_{E_6},z_{E_6}^2\} \,  , \  \text{where} \ \ z_{E_6} \ = \
  h_{\alpha_1}(\omega)h_{\alpha_3}(\omega^2)h_{\alpha_5}(\omega)h_{\alpha_6}(\omega^2)
  \,, \\
  Z(E_7^{sc}) \ \  & = \ \ \{e,z_{E_7}\} \,  , \   \text{where} \ \ z_{E_7} \ = \ h_{\alpha_2}(-1)h_{\alpha_5}(-1)h_{\alpha_7}(-1)\,,\\
\text{and~~}  Z(G_2) \ \ & = \ \ Z(F_4) \ \ = \ \ Z(E_8) \ \ = \ \ \{e\},
\endaligned
\end{equation}
with $\omega$  denoting a   primitive third root of unity.
Thus computations in the adjoint form can be performed by working modulo this center in the simply connected form described above.  (Since  $E_6^{ad}$ and $E_7^{ad}$ only appear in the guise of their complex points in this paper,  no problems arise from the fact that $\omega\notin\R$.)

Every Chevalley group $G$ has a Langlands dual Chevalley group $G^\vee$ whose root system is $\{\alpha^\vee|\alpha\in \Delta\}$, $\alpha^\vee=\frac{2}{( \alpha,\alpha)}\alpha$ (see Table~\ref{fig:dualgroups} and  Figure~\ref{fig:Dynkin}).  Passage to the dual interchanges long roots and short roots, so the numbering of the nodes for the Dynkin diagrams of ${\frak g}_2^\vee$ and ${\frak f}_4^\vee$ in Figure~\ref{fig:Dynkin} has been reversed.  By construction of $G^\vee$, the duality swaps the role of $X(T)$ and $X^\vee(T)$, and so elements of ${\frak a}^\vee_\C$ naturally give rise to elements of $\frak a^*_\C$, i.e., linear functionals.  We use the notation $T^\vee$   to denote   the  subgroup  corresponding to $T$   in $G^\vee$; it is the quotient, modulo the center, of the maximal torus  in the simply-connected algebraic covering group   described above.

In the rest of the paper we shall resume our assumption that $G$ is simply connected, keeping in mind that $G^\vee$ is in general not.

\section{Induced representations and flat sections}\label{sec:induced}

Fix a place $v\le \infty$ of $\Q$ and consider $T(\Q_v)\subset  B(\Q_v)\subset  G(\Q_v) $. 
  The valuation $|t^\alpha|_v$ is a positive real number for each $t\in T(\Q_v)$ and $\alpha\in \Delta$.  We  extend this definition to arbitrary elements $\lambda\in {\frak a}^*_{\C}=\Lambda_{\on{rt}} \otimes \C=\Lambda_{\on{wt}} \otimes \C$ by the formula
\begin{equation}\label{lambdaCaction}
  |t|_v^\lambda \ \ := \ \ \prod_{\alpha\in\Sigma} |t^\alpha|_v^{c_\alpha}\,,
\end{equation}
where  $\lambda$ is written as $\sum_{\alpha\in\Sigma} c_\alpha \alpha$, $c_\alpha \in \C$.  In particular, if $\lambda\in X(T)\subseteq \Lambda_{\on{wt}}$, then $|t|_v^\lambda =|t^\lambda|_v$.  This notation extends to $b=tn\in B(\Q_v)$ by defining $|b|_v^\lambda$ as $|t|_v^\lambda$, and further to  elements $b=\prod_{v\le \infty}b_v\in B(\A)$  (and hence $T(\A)$) of the global adele group $G(\A)$  through the product formula $|b|^\lambda=\prod_v|b|_v^\lambda$.  Each $b\in B(\A)$ determines a unique element $\log|b|\in\frak{a}$ such that
\begin{equation}\label{logbdef}
  e^{\langle \lambda,\log|b|\rangle} \ \ = \ \ |b|^\lambda
\end{equation}
for all $\lambda \in \frak{a}^*_\C$.

In addition to the continuous parameter $\lambda$, the induced representations studied in this paper also involve {\it quadratic} characters $\chi_v$ of $T(\Q_v)$, i.e., characters such that   $\chi_v^2$ is trivial.  We review two ways quadratic characters of $T(\Q_v)$ are connected to quadratic characters $\xi_v$ of $\Q_v^*$.  First,
composition of a fixed  $\xi_v$ with characters of $T$,
\begin{equation}\label{compwithxiv}
    \delta\,\in\,X(T) \ \ \longmapsto \ \ \xi_v\circ\delta\ ,
\end{equation}
defines a homomorphism from $X(T)$ to the group of quadratic characters of $T(\Q_v)$
 (and hence $B(\Q_v)$, after trivially extending to $N(\Q_v)$).   Since $\xi_v$ is quadratic, (\ref{compwithxiv}) factors through $X(T)/2X(T)$.  Under our assumption that $G$ is simply connected, $X(T)=\Lambda_{\on{wt}}$ and the character defined by (\ref{compwithxiv}) has the form $t\mapsto \xi_v(t^\delta)$ for some $\delta\in\Lambda_{\on{wt}}$.

Second, the duality $G\longleftrightarrow G^\vee$ identifies  $X(T^\vee)$ with $X^\vee(T)$.  Recall that since $G$ is simply-connected, elements of $T(\Q_v)$ have  unique factorizations of the form (\ref{uniquefact}).  Given a quadratic character $\xi_v$ of $\Q_v^*$ and an order two element $\sigma\in T^\vee(\C)$, we define a quadratic character $\chi_{v,\sigma}$ of $T(\Q_v)$ characterized by the property that
\begin{equation}\label{characterizechiv}
  \chi_{v,\sigma}\circ \varpi^\vee \ \ = \ \ \begin{cases}
\text{trivial}, & \sigma^{\varpi^\vee}=1\,\\
\xi_v, & \sigma^{\varpi^\vee} = -1\,,
\end{cases}
\end{equation}
where $\varpi^\vee$ is regarded as an element of $X^\vee(T)$ on the left-hand side and as an element of $X(T^\vee)$ on the right-hand side ($\chi_{v,\sigma}$ is uniquely determined by (\ref{characterizechiv})).
More concretely, the adjoint action of $\sigma$ gives an  involution of the Lie algebra ${\frak g}^\vee_\C$, with fixed point set $({\frak g}^\vee_\C)^\sigma$. Then for any   $\alpha\in \Delta_+$, $\chi_{v,\sigma}\circ \alpha^\vee$ is  trivial if  $X_{\alpha^\vee}\in ({\frak g}^\vee_\C)^\sigma$, and is    $\xi_v$ otherwise.  The involution   uniquely determines  $\sigma\in G^\vee(\C)$, since any   order-two element which induces the same involution
must differ by an element of the center of $G^\vee$, which is trivial for the adjoint group $G^\vee$.

The characters of interest in this paper
have the form $\chi_{v,\sigma}$ from (\ref{characterizechiv}) for particular order-two elements $\sigma\in T^\vee(\C)$.  Furthermore, they may be written using (\ref{compwithxiv}) as  $\chi_{v,\sigma}=\xi_v\circ \delta$, where a coset representative for $\delta=\delta(\sigma)\in  X(T)/2X(T)=\Lambda_{\on{wt}}/2\Lambda_{\on{wt}}$ can be taken to be the sum of the fundamental weights $\varpi_k$ for which $X_{\alpha_k^\vee}\notin ({\frak g}^\vee_\C)^\sigma$.
If $\xi=\prod_v \xi_v$ is a global quadratic character of $\Q^*\backslash \A^*$, then
 \begin{equation}\label{globalcompwithxi}
  \chi \ \ = \ \ \xi \circ \delta
\end{equation}
is the  global quadratic character   $\chi=\prod_v \chi_v$  of $T(\Q)\bs T(\A)$ whose local components are related by (\ref{compwithxiv}) and (\ref{characterizechiv}).

 Let $\rho=\frac{1}{2}\sum_{\alpha\in \Delta_+}\alpha$.
 For $v$ a  place of $\Q$, $\lambda \in {\frak a}_{\C}^*$, and $\chi_v$ a quadratic character of $T(\Q_v)$,
 let
 \begin{equation}\label{Inudef}
    I_v(\lambda,\chi_v) \ \ = \ \ \Ind_{B(\Q_v)}^{G(\Q_v)} \chi_v |\cdot |_v^\lambda
\end{equation}
 denote the principal series representation of $G(\Q_v)$ unitarily induced from $\chi_v |\cdot |_v^\lambda$; by this we mean the  vector space of $K_v$-finite functions $f:G(\Q_v) \to \C$ satisfying the transformation law
 \begin{equation}\label{Inuspace}
     f(bg) \, = \,  \chi_v(b)  \, |b|_v^{\lambda + \rho} \, f(g)\, , \ \
 \text{for all~} \,b \in B( \Q_v) \ \,  \text{and} \,  \ g \in G(\Q_v) \,,
 \end{equation}
 on which $G(\Q_v)$ acts by right translation for $v<\infty$. (When $v=\infty$ this definition instead produces the $(\frak g,K)$-module of a principal series representation.)
  Each function in $I_v(\lambda,\chi_v)$ is determined by its restriction to the maximal compact subgroup $K_v$ via the  Iwasawa decomposition
$G(\Q_v) = B(\Q_v)  K_v$.  Note that the Iwasawa decomposition of an element of $G(\Q_v)$ is not unique, owing to the nontrivial intersection of $B(\Q_v)$ and $K_v=G(\Z_v)$;~thus these restrictions
must satisfy
\begin{equation}\label{KTcompatlocal}
f(bk) \ = \ \ \chi_v(b)\, f(k)\, , \ \ \ \text{for all~}\, b\,\in\,B(\Q_v)\cap K_v\ \, \text{and} \, \ k\,\in\,K_v\,,
\end{equation}
 in order to be compatible with the transformation law    (\ref{Inuspace}).
 Restriction to $K_v$ provides an isomorphism of $K_v$-modules  between $I_v(\lambda,\chi_v)$ and $\Ind_{B(\Q_v) \cap K_v}^{K_v} \chi_v$, the space of $K_v$-finite functions satisfying (\ref{KTcompatlocal}) endowed with the right translation action of $K_v$.  A family of functions  $\{g\mapsto f(g,\lambda)\in I_v(\lambda,\chi_v)|\lambda\in {\frak a}_\C^*\}$ having a common restriction to $K_v$ is called a {\it local flat section}.

Likewise, for a global quadratic character $\chi=\prod_v\chi_v$ as above let
 \begin{equation}\label{Idef}
    I(\lambda,\chi) \ \ = \ \ \Ind_{B(\A)}^{G(\A)} \chi |\cdot |^\lambda
\end{equation}
denote the vector space of all finite linear combinations of pure tensors $f=\otimes_{v\le\infty} f_v$,   $f_v\in I_v(\lambda,\chi_v)$, for which $f_v|_{K_v}\equiv 1$ for  all but finitely many places $v<\infty$.  Thus
\begin{equation}\label{globalprincseries}
f(bg) \ \ =  \ \ \chi(b) \,  |b|^{\lambda + \rho}\,f(g) \, ,
\ \ \  \text{for all~}\, b \in B( \A)\ \text{and} \ g \in G(\A)\,.
\end{equation}
Here, too, $f$ is determined by its restriction to the maximal compact subgroup $K=\prod_{v\le\infty}K_v$, which must satisfy
\begin{equation}\label{KTcompatglobal}
    f(bk) \ = \ \ \chi(b)\, f(k)\, , \ \ \ \text{for all~}\,b\,\in\,B(\A)\cap K\ \text{and} \ k\,\in\,K\,;
\end{equation}
furthermore, $f$ is $K$-finite under right-translation.
 Similarly to the local situation, we write  $\Ind_{B(\A) \cap K}^K \chi$
 for the set of functions on $K$ obtained by restricting elements of $I(\lambda,\chi)$ to $K$, and call a family of functions  $\{g\mapsto f(g,\lambda)\in I(\lambda,\chi)|\lambda\in {\frak a}_\C^*\}$ a {\it global flat section} if they share a common restriction to $K$.

\section{Unipotent Arthur parameters}
\label{sec:unipotentarthurparmeters}

Let $W_\R$ denote the Weil group of $\R$, which is isomorphic to the  subgroup $\C^*\cup j \C^*$ of the multiplicative quaternions.  An  Arthur parameter for $G(\R)$ is a homomorphism
\begin{equation}\label{Arthur1}
    \psi \ : \  W_\R\,\times \,SL(2,\C)  \ \ \longrightarrow \ \ G^\vee(\C)
\end{equation}
whose restriction to $SL(2,\C)$ is algebraic, and whose restriction to $W_\R$ satisfies some further conditions.  Those conditions are automatically met in what Arthur \cite[p.~28]{Arthur-conjectures2} cites as the most difficult  case:~the {\it unipotent} Arthur parameters, where $\psi$'s restriction to the index-two subgroup $\C^*$ of $W_\R$ is trivial.   Thus a unipotent Arthur parameter is determined by $\psi|_{SL(2,\C)}$ and $\sigma=\psi(j)$, the latter having order at most two (because $\psi$ is trivial on $j^2=-1\in \C^*$).  Since $\psi$ in (\ref{Arthur1}) is a homomorphism of a direct product, the image $\psi(SL(2,\C))$ must commute with $\sigma$.  At the level of Lie algebras this
 commutativity is equivalent  to insisting that the values of the differential $d\psi$ of $\psi$ on the standard basis   $\{\ttwo0100,\ttwo 0010,\ttwo{1}00{-1}\}$ of ${\frak{sl}}(2,\C)$ lie inside the symmetric Lie subalgebra $({\frak g}^\vee_\C)^{\sigma}$ of $\frak{g}_\C^\vee$ fixed by $\sigma$ under the adjoint action.

  Since Arthur's conjectures concern $G^\vee(\C)$-conjugacy classes of Arthur parameters, there is no loss of
 generality in assuming that the semisimple element $\sigma=\psi(j)$ lies in $T^\vee(\C)$ and that  $d\psi\ttwo{1}00{-1}$ lies in $\frak{a}^\vee$, which we recall is identified with  $\frak a^*$, the $\R$-span of $\Lambda_{\on{rt}}$.  After taking a yet another   conjugate by the Weyl group if necessary, we may assume that $d\psi\ttwo{1}00{-1} \in {\frak a}^*$ is dominant and write it as $2\lambda_{{\cal O}^\vee}$, where
   $\mathcal O^\vee\subset \frak g_\C^\vee$ is the  adjoint nilpotent orbit of  $G^\vee(\C)$ containing $d\psi\ttwo0100$.  (Note that this conjugation potentially alters $\sigma=\psi(j)$, however.)  The pairings of $2\lambda_{\mathcal O^\vee}$ with the simple roots for $G^\vee$ (equivalently, simple coroots for $G$) are the nonnegative integers determining the weighted Dynkin diagram of $\mathcal O^\vee$ \cite{collingwood}, so in fact $d\psi\ttwo{1}00{-1} =2\lambda_{{\cal O}^\vee}$ is a dominant integral weight.

Thus a real unipotent Arthur parameter is determined by an element  $\sigma=\psi(j)\in T^\vee(\C)$ having order at most two, and a complex adjoint nilpotent orbit $\cal O^\vee\subset \frak{g}^\vee_\C$ which intersects $({\frak g}^\vee_\C)^{\sigma}$ nontrivially. The Langlands parameter attached to $\psi$ is the homomorphism $\psi_{\text{Langlands}}:W_\R\rightarrow G^\vee(\C)$ given by
\begin{equation}\label{langlandsparam}
\psi_{\text{Langlands}} \ : \ \srel{z\,\in\,\C^*  \ \  \longmapsto \ \ \psi\ttwo{|z|}00{|z|^{-1}} }{ j  \ \   \longmapsto \ \  \sigma\,=\,\psi(j)\,.}
\end{equation}
  We   assume $\sigma$ has order exactly 2, as the case of $\sigma$ trivial was treated in \cite{M}. Given a global character $\xi$ of $\Q^*\backslash \A^*$, we explained in section~\ref{sec:induced} how to construct a character $\chi_\sigma=\otimes_v\chi_{\sigma,v}$ of $T(\Q)\backslash T(\A)$ from $\sigma=\psi(j)$ using (\ref{characterizechiv}).


\begin{thm}\label{thm:unitarizablequotient}
Let $G$ be a Chevalley group other than $Spin(n+1,n)$, $Spin(n,n)$, or $HSpin(2n)$, and
let $\psi$ be a real unipotent Arthur parameter as above.
Then the  principal series  $I_\infty(\lambda_{{\cal O}^\vee},\chi_{\sigma,\infty})$ has a  unitarizable quotient having Langlands parameter $\psi_{\text{Langlands}}$.
\end{thm}

This statement is (mostly) a special case of Arthur's conjectures \cite{Arthur-conjectures1,Arthur-conjectures2}, which  are   stronger, global statements   concerning a full packet of representations (potentially more than just this one having Langlands parameter $\psi_{\text{Langlands}}$).  As mentioned in the introduction, the full Arthur packets are very difficult to identify.
Our contribution to Theorem~\ref{thm:unitarizablequotient} is for the exceptional groups $G$, since the classical group cases have been settled by M\oe glin \cite{Moeglin} and Arthur \cite{Arthur-book} (in addition, the case of $G_2$ was established by  Vogan \cite{Vogan}; see also \cite{Kim-G2}).  Recall from Section~\ref{sec:chevalleygroups} that it suffices to establish Theorem~\ref{thm:unitarizablequotient} for simply connected $G$.

\subsection{Reduction to distinguished Arthur parameters}\label{sec:sub:reductiontodistinguished}

Arthur points out in \cite[pp.~43-44]{Arthur-conjectures2} that when the image of an Arthur parameter $\psi$ is contained in a proper Levi subgroup of $G^\vee$ (i.e., $\psi$ is not ``distinguished'', or not ``elliptic'' in Arthur's terminology), one expects the full Arthur packet for $\psi$ to be unitarily induced from the Arthur packet for the corresponding parameter on the corresponding proper Levi subgroup.  In particular, Theorem~\ref{thm:unitarizablequotient} for such non-distinguished parameters $\psi$ can in principle be shown inductively from smaller groups.  In practice this requires calculating the composition series of the induced representation.

 In our present setting of exceptional groups it is possible to list all non-distinguished Arthur parameters, and directly verify Theorem~\ref{thm:unitarizablequotient} for them in {\tt atlas} \cite{atlas} -- either by directly invoking its {\tt is\_unitary} command (which works in many instances, but is impractical for some representations on $E_7$ and $E_8$) or by verifying Arthur's induction prediction from the previous paragraph.  We carried this out directly; as a result  Theorem~\ref{thm:unitarizablequotient} reduces to the case that $\psi$ is distinguished.
%
     The next two sections accordingly enumerate the full (finite) list of distinguished Arthur parameters  for  exceptional groups.

\section{Order two elements and stabilizers}\label{sec:ordertwoelements}

Because of their importance in describing the real unipotent Arthur parameters in Section~\ref{sec:unipotentarthurparmeters}, we shall now describe Elie Cartan's classification of the conjugacy classes of order two elements $\sigma$ in the adjoint  group $G^\vee(\C)$, as well as  the symmetric Lie subalgebras $({\frak g}^\vee_\C)^{\sigma}$ of $\frak{g}_\C^\vee$ they fix under the adjoint action.  We refer to \cite{Reeder} for a recent exposition of Kac's classification of finite-order Lie algebra automorphisms.

 Recall that we assume $G$ is simply connected and exceptional.
 It follows that
  $G^\vee(\C)$ is isomorphic to the quotient of $G(\C)$ by its center (\ref{centers}), and hence shares the common Lie algebra $\frak g_\C$; the identification of $\frak g_\C$ and $\frak g_\C^\vee$ sometimes requires keeping into account the reversed   numbering of the Dynkin diagram (see  Figure~\ref{fig:Dynkin}).    Without loss of generality we again assume  that
the semisimple element $\sigma$ lies in $T^\vee(\C)$.

Cartan's classification can be cleanly stated in terms of
the affine Dynkin diagram of $\f g_\C\cong \f g_\C^\vee$, which is obtained by
adding an additional node,   labelled $\mbox{\textcircled{\tiny 0}}$, to the Dynkin diagram of $\f g_\C$.  This node corresponds
to $-\alpha_{\text{high}}$, where $\alpha_{\text{high}}$ is the highest root of $\f g_\C$, and   is joined to the node  labelled $\mbox{\textcircled{\tiny i}}$ (corresponding to the simple root
  $\alpha_i$) if and only if $-\alpha_{\text{high}} + \alpha_i\in \Delta$  (equivalently, if $\alpha_i$ is not orthogonal to $\alpha_{\text{high}}$).  There is a unique $\alpha_i\in \Sigma$ with this property in each of the five exceptional root systems
(indeed, in all simple root systems except $A_n,$ $n>1$).
Figure~\ref{fig:Dynkin} lists the affine Dynkin diagrams for each exceptional root system.  In particular, this simple root is $\alpha_1$ in the cases of $\frak g_2$, $\frak f_4$, and $\frak e_7$, whereas it is $\alpha_2$ for $\frak e_6$ and $\alpha_8$ for $\frak e_8$.

Deleting any nonempty subset $S$ of nodes from the affine Dynkin diagram
results in the Dynkin diagram of some
reductive complex subalgebra $\f g'_\C$ of $\f g_\C$,
generated by   all   $X_{\alpha_i}$ and $X_{-\alpha_i}$ corresponding to the
undeleted nodes $\mbox{\textcircled{\tiny i}}$; thus $\frak g'$ has semisimple rank $r+1-\#S$.  If only one node is deleted, $\f g'_\C$ and $\f g_\C$ have equal rank and  share the common Cartan subalgebra $\frak a_\C$, with respect to which $\f g'_\C$ has a base of simple roots
 corresponding to the undeleted nodes.  Of course $\f g'_\C=\f g_\C$ when node $\mbox{\textcircled{\tiny 0}}$ is deleted, but otherwise this base of
simple roots for $\f g'_\C$ is not compatible with the notion of positivity
  inherited from $\f g_\C$.

Write $\alpha_{\text{high}}=\sum_{i\le r}n_i\alpha_i$; $n_i$ is the integer given just above and to the left of the node $\mbox{\textcircled{\tiny i}}$ in Figure~\ref{fig:Dynkin}.
\begin{figure}$$\text{
\includegraphics[scale=.339]{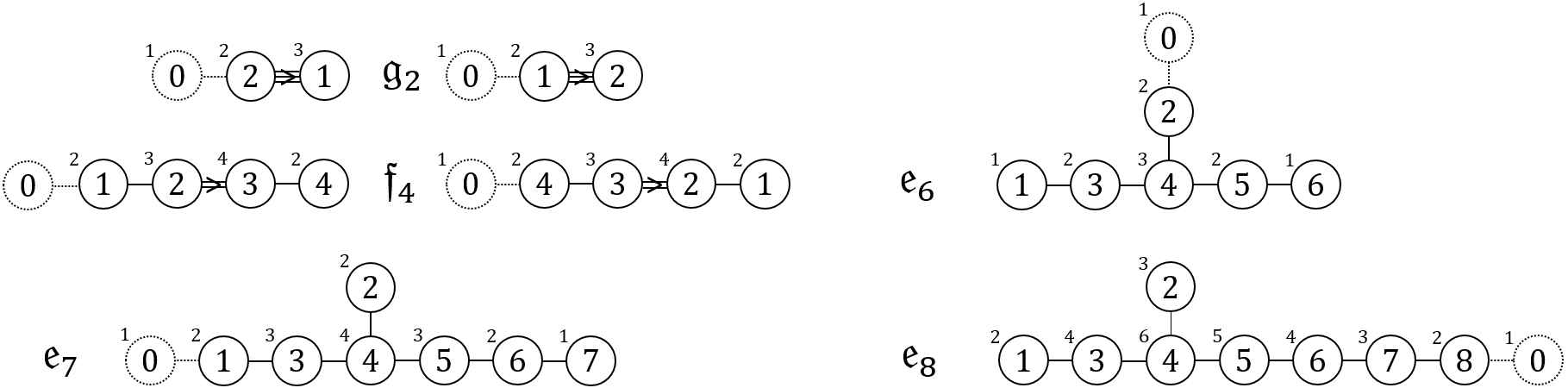}}$$
\caption{The extended Dynkin diagrams of the exceptional simple complex Lie algebras (left) and their duals  $\frak g_\C^\vee$ (right),
 which coincide for $\frak e_6$, $\frak e_7$, and $\frak e_8$.  Aside from the diagrams  for the duals of $\frak g_2$ and $\frak f_4$ (on the right of the first column), the positive integers inside each circle represent the usual Bourbaki numbering of the simple roots.  Those dual diagrams    involve reversing the arrows (owing to the interchange of short roots and long roots), hence the change in numbering.   The dotted circle containing 0 represents the affine node that is added to the classical Dynkin diagram, via a bond marked with a dotted line.  The integers above and to the left of each solid circle represent the coefficient of the highest  root when expanded as a positive integral combination of the simple roots.  They play an important role in  Cartan's classification of order two elements. }
\label{fig:Dynkin}
\end{figure}
  Cartan showed that if $\sigma\in T^\vee(\C)$ has order two, then  -- after possibly replacing $\sigma$ by a Weyl conjugate --
${\frak g}^\sigma_\C=\frak g_\C'+ \frak a_\C$, where $\frak g'_\C$   is obtained from the  affinize-and-delete procedure in the previous paragraph by either
\begin{enumerate}
\item[1)] removing a single node of the form $^2\mbox{\textcircled{\tiny i}}$, or
\item[2)] removing two nodes of the form $^1\mbox{\textcircled{\tiny i}}$
\end{enumerate}
from   the extended Dynkin diagram for $\f g_\C$ in Figure~\ref{fig:Dynkin} (here possibly $i=0$).  Conversely, for each subalgebra
$\frak g'_\C$
 obtained through either 1) or 2) there exists  a unique\footnote{The uniqueness follows from the fact that $G^\vee(\C)$ has trivial center, and hence only the identity element acts trivially on $\frak g_\C$.  Note that
 preimages of
  $\sigma$ in the covering group $G(\C)$  might have higher order, i.e., their squares might be nontrivial central elements.}  order-two element $\sigma\in G^\vee(\C)$
such that  $\f g^\sigma_\C=\frak g_\C'+ \frak a_\C$.  Note that  different choices of $S$ might yield order two elements $\sigma$ in the same Weyl group orbit, which can be straightforwardly detected.

Table~\ref{table:listofgjs} gives a  list of all conjugacy classes order two elements in Cartan's classification  for the dual algebras $\frak g^\vee_\C$ of the five   exceptional complex Lie algebras $\frak g_\C$.
\begin{table}$$\text{
\begin{tabular}{|c|c|c|c|c|}
\hline
$\frak g_\C^\vee$ & $(\frak g_\C^\vee)^\sigma$ & Proper Levi? & Nodes deleted & $\tilde\sigma\in T(\C)\subset G(\C)$\\
\hline
${\frak g}_2$ &  $\frak{sl}_2\oplus \frak{sl}_2$ & no & $\mbox{\textcircled{\tiny 1}}$  & $h_{\alpha_1}(-1)$ \\
$\frak{f}_4$ &  $\frak{sp}_6\oplus\frak{sl}_2$ & no  & $\mbox{\textcircled{\tiny 4}}$ & $h_{\alpha_2}(-1)h_{\alpha_4}(-1)$ \\
$\frak{f}_4$ &  $\frak{so}_9$ & no  & $\mbox{\textcircled{\tiny 1}}$ &  $h_{\alpha_3}(-1)$\\
$\frak{e}_6$ &  $\frak{so}_{10}+\frak a_\C$ & yes   &  $\mbox{\textcircled{\tiny 0}},\mbox{\textcircled{\tiny 6}}$&  $h_{\alpha_2}(-1)h_{\alpha_5}(-1)$  \\
$\frak{e}_6$ & $\frak{sl}_{6}\oplus\frak{sl}_{2}$ & no  &  $\mbox{\textcircled{\tiny 2}}$& $h_{\alpha_1}(-1)h_{\alpha_4}(-1)h_{\alpha_6}(-1)$  \\
$\frak{e}_7$ &   $\frak{e}_6+\frak a_\C$ & yes  & $\mbox{\textcircled{\tiny 0}},\mbox{\textcircled{\tiny 7}}$& $h_{\alpha_1}(-1) h_{\alpha_2}(i) h_{\alpha_4}(-1) h_{\alpha_5}(-i) h_{\alpha_7}(i)$ \\
$\frak{e}_7$ &  $\frak{sl}_8$ & no &  $\mbox{\textcircled{\tiny 2}}$ &
$h_{\alpha_2}(i)h_{\alpha_5}(-i)h_{\alpha_6}(-1)h_{\alpha_7}(i)$ \\
$\frak{e}_7$ &  $\frak{so}_{12}\oplus\frak{sl}_2$ & no  &  $\mbox{\textcircled{\tiny 1}}$& $h_{\alpha_2}(-1)h_{\alpha_3}(-1)$ \\
$\frak{e}_8$ &   $\frak{so}_{16}$ & no   &  $\mbox{\textcircled{\tiny 1}}$  & $h_{\alpha_2}(-1)h_{\alpha_3}(-1)$  \\
$\frak{e}_8$ &   $\frak{e}_7\oplus\frak{sl}_2$ & no  &    $\mbox{\textcircled{\tiny 8}}$ & $h_{\alpha_2}(-1)h_{\alpha_5}(-1)h_{\alpha_7}(-1)$ \\
\hline
\end{tabular}}$$
\caption{Elie Cartan's classification of conjugacy classes of order two elements in the complex exceptional adjoint groups $G^\vee(\C)$.    The circled numbers describing the deleted nodes refer to the dual Lie algebra Dynkin diagrams for $\frak g_\C^\vee$ (e.g., on the right in Figure~\ref{fig:Dynkin}).
The last column lists an element of $T(\C)\subset G(\C)$ which projects modulo the center (\ref{centers}) to $\sigma$; here $G(\C)$ is viewed as the cover of $G^\vee(\C)$, and so in this last column we instead use the numbering of the nodes from the diagrams on the left in Figure~\ref{fig:Dynkin}. }
\label{table:listofgjs}
\end{table}
Two cases involve deleting the node $^1\mbox{\textcircled{\tiny 0}}$, so that $(\frak g_\C^\vee)^\sigma$ is a proper Levi subalgebra; these two cases cannot give rise to distinguished Arthur parameters, and so will be ignored because of the reduction in Section~\ref{sec:sub:reductiontodistinguished}.
   The fourth column describes which nodes are deleted to form $(\frak g_\C^\vee)^\sigma$.  Under our  identification of $\frak g^\vee_\C$ with $\frak g_\C$,  $(\frak g_\C^\vee)^\sigma\subset  \frak g^\vee_\C$ corresponds to $(\frak g_\C )^{\tilde\sigma}\subset  \frak g_\C$ where $\tilde\sigma \in T(\C)$ is listed in the last column.  The elements  $\sigma$ and $\tilde\sigma$ agree for the   $E_8$ examples, and are related by reversing the numbering of the Dynkin diagram in the $G_2$ and $F_4$ examples.  That leaves the $E_6$ and $E_7$ examples, in which we have listed a preimage $\tilde\sigma$ of $\sigma$ in the simply connected cover.

%
%
%
%
%
%

\section{Orbits and Arthur's conjectures}
\label{sec:listofcases}

In Table~\ref{table:listofgjs} we listed the conjugacy classes of order two   elements $\sigma\in G^\vee(\C)$ and the symmetric subalgebras $({\frak g}^\vee_\C)^\sigma$ they fix.
In all but two entries in Table~\ref{table:listofgjs} (both of which have already been disregarded since  $({\frak g}^\vee_\C)^\sigma$ is a proper Levi subalgebra, and hence cannot contribute any interesting cases in light of the reduction in Section~\ref{sec:sub:reductiontodistinguished}), $({\frak g}^\vee_\C)^\sigma$ is constructed using the affinize-and-delete procedure reviewed in Section~\ref{sec:ordertwoelements} by removing a single node marked with a 2 above and to the left of it in Figure~\ref{fig:Dynkin}. It is an equal-rank reductive subalgebra of $\frak g^\vee_\C$ sharing the common Cartan $\frak a^\vee_\C$.  Let $\beta_1,\ldots,\beta_r$ be the base of simple roots (in some choice of numbering) for $({\frak g}^\vee_\C)^\sigma$ corresponding to the undeleted nodes in the affinize-and-delete procedure.  That is, there exists an index $k\in\{1,\ldots,r\}$ and a permutation $\tau\in S_r$ such that for $H\in \frak{a}^\vee_\C$,
\begin{equation}\label{subbase}
\beta_k(H)\,=\,-\alpha_{\text{high}}^\vee(H) \ \ \ \ \text{and} \ \ \  \ \beta_i(H)\,=\,\alpha_{\tau(i)}^\vee(H) \ \ \text{for} \ i\neq k\,,
\end{equation}
where $\alpha_{\text{high}}^\vee$ denotes the highest root for $\frak g^\vee_\C$ with respect to its basis of simple roots $\{\alpha_1^\vee,\ldots,\alpha_r^\vee\}$.

 We next turn to the calculation of the     conjugacy classes of  real unipotent Arthur parameters (\ref{Arthur1}) $\psi$ relevant for Theorem~\ref{thm:unitarizablequotient}, which will be listed in Table~\ref{fig:biglistofcases}.
   They are determined by $\sigma$ as well as  a complex adjoint nilpotent orbit ${\cal O}^\vee\subset {\frak g}^\vee _\C$ of $G^\vee(\C)$ that intersects $({\frak g}^\vee_\C)^\sigma$.  Recall from Section~\ref{sec:sub:reductiontodistinguished} that we have reduced to the situation that  $\psi$'s image is not contained in a proper Levi subgroup.   Furthermore, it can easily be  seen   (by writing  Levis as centralizers of elements in $\frak a_\C^\vee$) that the proper Levi subalgebras of $({\frak g}^\vee_\C)^\sigma$ are always contained in proper Levi subalgebras of ${\frak g}^\vee _\C$.  In particular, if  $d\psi\ttwo0100$ lies in a proper Levi subalgebra of $({\frak g}^\vee_\C)^\sigma$, then by the Jacobson-Morozov Theorem the image $\psi(SL(2,\C))$ is contained in a proper Levi.\footnote{Note that the image of the Arthur parameter $\psi:W_\R\times SL(2,\C)\rightarrow \G^\vee(\C)$ may nevertheless not be contained in this proper Levi, since $\psi(j)$ might lie outside it.}
    Thus the following procedure  determines all pairs of $(\sigma,\lambda_{\mathcal O^\vee})$ necessary to prove Theorem~\ref{thm:unitarizablequotient}.
\begin{enumerate}
  \item List all  distinguished, complex adjoint nilpotent orbits $\mathcal O^\vee_\sigma$ of the reductive Lie subalgebra $({\frak g}^\vee_\C)^\sigma$.
  \item Consider an Arthur parameter $\psi$ from (\ref{Arthur1}) with $d\psi\ttwo 0100 \in \mathcal O^\vee_\sigma\subset ({\frak g}^\vee_\C)^\sigma\subset{\frak g}^\vee_\C  $.
       The Lie algebra $({\frak g}^\vee_\C)^\sigma$ exponentiates to a subgroup $H_\sigma$ of $G^\vee(\C)$ which commutes with $\sigma$.
       By replacing $\psi$ with some $H_\sigma$-conjugate if necessary, arrange that $d\psi\ttwo{1}00{-1}$ lies in the Cartan $\frak{a}_\C^\vee$ common to both $({\frak g}^\vee_\C)^\sigma$ and ${\frak g}^\vee_\C $,  and is furthermore dominant with respect to the base of simple roots $\beta_1,\ldots,\beta_r$ of $({\frak g}^\vee_\C)^\sigma$.  The weighted Dynkin diagram of $\mathcal O^\vee_\sigma\subset  ({\frak g}^\vee_\C)^\sigma$  is labeled with the values of $\beta_i(d\psi\ttwo{1}00{-1})$, which  are either 0 or 2 for distinguished orbits \cite[Theorem 8.2.3]    {collingwood}.
Use this information  and (\ref{subbase}) to compute the values of the simple coroots $\alpha_i^\vee$ on $d\psi\ttwo{1}00{-1}$, which then determine the element  $2\lambda_{\mathcal O^\vee_\sigma}\in {\frak a}^*$ satisfying $\langle 2\lambda_{\mathcal O^\vee_\sigma},\alpha_i^\vee\rangle = \alpha_i^\vee(d\psi\ttwo{1}00{-1})$, $i\le r$.
Let $\lambda_1=-\lambda_{\mathcal O^\vee_\sigma}$, which in all of our cases is equal to $w_{\text{long}}\lambda_{\mathcal O^\vee_\sigma}$.
\item  Let $\mathcal O^\vee\subset \frak g^\vee_\C$ denote the $G^\vee(\C)$-saturation of $\mathcal O^\vee_\sigma$ (i.e., its containing orbit in $\frak{g}^\vee_{\C}$).  By construction, there exists a Weyl element $w_0\in W$ such that $\lambda_0:=w_0^{-1}\lambda_1$ lies in the positive Weyl chamber of $\frak a^*$, and such that the values $\langle 2\lambda_0,\alpha_i^\vee\rangle$ label the weighted Dynkin diagram of $\mathcal O^\vee$.  Then $\lambda_0=\lambda_{\mathcal O^\vee}$ is a dominant integral weight.
\end{enumerate}

\noindent {\bf Example.} To illustrate the above procedure, consider the case of $(\frak g^\vee_\C,(\frak g^\vee_\C)^\sigma)=(\frak e_7,\frak{sl}_{8})$ in   Table~\ref{table:listofgjs}.  There is precisely one distinguished nilpotent orbit of $\frak{sl}_{8}$:~the {\it principal} orbit, corresponding to a full Jordan block.  It contains the sums  of nonzero root vectors for all seven simple roots, and thus under the affinize-and-delete  embedding we deduce $(\frak g^\vee_\C)^\sigma$ includes the element $X_{-\alpha_{\text{high}}}+X_{\alpha_1}+X_{\alpha_3}+X_{\alpha_4}+X_{\alpha_5}+X_{\alpha_6}+X_{\alpha_7}\in \frak e_7$.  It is straightforward to see that $\lambda_1=9\varpi_2-\rho\in \Lambda_{\on{wt}}$, since $\langle2\lambda_1,-\alpha^\vee_{\text{high}}\rangle =\langle2\lambda_1,\alpha^\vee_{i}\rangle=-2$ for $i\in \{1,3,4,5,6,7\}$. The Weyl orbit of $\lambda_1$ includes the dominant integral weight $\lambda_0=\varpi_1+\varpi_4+\varpi_6$, so that $2\lambda_0$ corresponds to the weighted Dynkin diagram 2002020 for the nilpotent orbit of $\frak e_7$ with Bala-Carter label $E_6(a_1)$.  Although this orbit is not distinguished, the Arthur parameter $\psi$ is nevertheless distinguished:~the image of $SL(2,\C)$ under $\psi$ is contained in a proper Levi, but the full image of $\psi$ is not.

Table~\ref{fig:biglistofcases} lists all distinguished Arthur parameters $\psi$ (with $\psi(j)$ nontrivial) for exceptional groups, which (by the reduction of Section~\ref{sec:sub:reductiontodistinguished}) suffices   to prove Theorem~\ref{thm:unitarizablequotient}.
In all but three cases (the above $E_6(a_1)$ example   for $E_7$, and the first two listed cases for $E_8$) both the parameter $\psi$ and the orbit $\mathcal O^\vee$ are distinguished.  In all other cases of non-distinguished orbits $\mathcal O^\vee \subset (\frak g_\C^\vee)^\sigma$ we were able to show the parameter is also not distinguished, by finding a nontrivial torus inside $T(\C)$ commuting with $\sigma$.  In these remaining three listed cases, the connected component of the centralizer of $\psi(SL(2,\C))$ is an explicit torus on which conjugation by $\sigma$ acts as inversion; in particular, the full image of $\psi$ cannot be contained in any proper Levi subgroup.

 The first    column  in Table~\ref{fig:biglistofcases} describes $\sigma=\psi(j)$ implicitly, and the second lists the orbit  $\cal O^\vee$.
The remaining columns  describe  the principal series parameters in
 Theorem~\ref{thm:unitarizablequotient}.  Here as in Section~\ref{sec:induced} the order two element $\sigma\in T^\vee(\C)$ is described by an element $\delta(\sigma)\in \Lambda_{\on{wt}}/2\Lambda_{\on{wt}}$ for which $(-1)^{\langle \delta(\sigma),\alpha_i^\vee\rangle}=\sigma^{\alpha_i^\vee}$, $i=1,\ldots,r$, where $\alpha_i^\vee$ is regarded as a coroot on the left-hand side and as an element of $X(T^\vee)$ on the right-hand side.   The second pair of columns lists $\lambda_1$ determined as above, along with   $\delta_1=\delta(\sigma)$.  In each case $\delta_1$ can be taken to be the fundamental weight corresponding to the deleted node in the affinize-and-delete procedure.  The third pair of columns lists a chosen simultaneous Weyl conjugate $(\lambda_0,\delta_0 \pmod{2\Lambda_{\on{wt}}})=(w_0^{-1}\lambda_1,w_0^{-1}\delta_1\pmod{2\Lambda_{\on{wt}}})$, for which $\lambda_0$ is dominant.

\begin{table}[h]
$$
\begin{array}{|c|c||c|c||c|c|}
\hline
\text{$(G,  (\frak g_\C^\vee)^\sigma)$}&
\text{$\cal O^\vee$}&
\lam_1&\delta_1 &\lam_0& \delta_0\\
\hline
(G_2,\ \f{sl}_2\oplus\f{sl}_2)&G_2(a_1) & [2,-1] & \varpi_1 & 10 &   11\\
\hline
(F_4,\ \f{sp}_6 \oplus \f{sl}_2) &F_4(a_3)&[-1,0,-1,3]&\varpi_4&0010&0101\\
(F_4,\ \f{sp}_6 \oplus \f{sl}_2 )&F_4(a_2)&[-1,-1,-1,5]&\varpi_4&1010
& 0101\\
(F_4,\ \f{so}_9) &F_4(a_3)&[2,0,-1,0]&\varpi_1&0010&  1000\\
(F_4,\ \f{so}_9) &F_4(a_1)&[5,-1,-1,-1]&\varpi_1&1011&  1100\\
\hline
(E_6^{sc}, \  \f{sl}_6 \oplus \f{sl}_2) & E_6(a_3)&[-1,5,-1,-1,-1,-1]&\varpi_2&100101& 011010\\
\hline
(E_7^{sc},\ \f{so}_{12} \oplus \f{sl}_2) & E_7(a_5)&
-[-5,1,1,0,1,0,1]&\varpi_1&0001001&1101011
\\
(E_7^{sc},\ \f{so}_{12} \oplus \f{sl}_2) & E_7(a_4)&
-[-6,1,1,0,1,1,1]&\varpi_1& 1001001&1111011
\\
(E_7^{sc},\ \f{so}_{12} \oplus \f{sl}_2) & E_7(a_3)
&-[-8,1,1,1,1,1,1]&\varpi_1&1001011&1111110
\\
(E_7^{sc},\ \f{sl}_{8})  & E_6(a_1)
&-[1,-8,1,1,1,1,1]&\varpi_2&1001010&1111111
\\
\hline
(E_8,\ \f{e}_7\oplus \f{sl}_2)& D_5+A_2&
-[1,0,0,1,0,0,1,-6]& \varpi_8 & 00001001 & 00010100
\\
(E_8,\ \f{e}_7\oplus \f{sl}_2)& D_7(a_1)&
-[1,0,0,1,0,1,1,-8]& \varpi_8 & 10001001 & 00010100
\\
(E_8,\ \f{e}_7\oplus \f{sl}_2)& E_8(a_7)&
[0,0,0,-1,0,0,-1,5]&\varpi_8&\varpi_5&00010100
\\
(E_8,\ \f{e}_7\oplus \f{sl}_2)& E_8(b_5)&
-[1,1,1,0,1,0,1,  -9 ]&\varpi_8&00010011&01101000
\\
(E_8,\ \f{e}_7\oplus \f{sl}_2)& E_8(b_4)&
-[1,1,1,0,1,1,1,-11]&\varpi_8&10010011&01101000
\\
(E_8,\ \f{e}_7\oplus \f{sl}_2)& E_8(a_3)&
-[1,1,1,1,1,1,1,-14]& \varpi_8 &10010111&01101000
\\
(E_8,\ \f{so}_{16})& E_8(a_7)&
-[-5,0,0,1,0,0,1,0]&\varpi_1&\varpi_5&  00101010\\
(E_8,\ \f{so}_{16})& E_8(b_6)&
-[-8,1,1,0,1,0,1,0]&\varpi_1&00010001&  11010101\\
(E_8,\ \f{so}_{16})& E_8(a_6)&
-[-9,1,1,0,1,0,1,1]&\varpi_1&00010010&  11010111\\
(E_8,\ \f{so}_{16})& E_8(a_5)&
-[-11,1,1,0,1,1,1,1]&\varpi_1&10010010&  11110111\\
(E_8,\ \f{so}_{16})& E_8(a_4)&-[-14,1,1,1,1,1,1,1]&\varpi_1&10010101&  11111111\\
\hline
\end{array}$$
\caption{A list of cases needed in order to prove Theorem~\ref{thm:unitarizablequotient}.  Here $G$ is a simply connected form from Table~\ref{fig:dualgroups},
$(\frak g_\C^\vee)^\sigma$ is one of the symmetric subalgebras from Table~\ref{table:listofgjs}, and $\mathcal O^\vee$ is an adjoint orbit in $\frak g_\C^\vee$ meeting  $(\frak g_\C^\vee)^\sigma$ (identified by its Bala-Carter label).  The pair $(\lambda_1,\delta_1)$ encodes the Arthur parameter; $(\lambda_0,\delta_0)$ is a simultaneous Weyl translate $(\lambda_1,\delta_1)$ for which $\lambda_0=\lambda_{\mathcal O^\vee}$ is a dominant weight.  To save space, weights are listed as a vector or string encoding the coefficients in its expansion as a sum of the fundamental weights $\varpi_1,\ldots,\varpi_r$.
 }
\label{fig:biglistofcases}
\end{table}

\section{Quadratic Dirichlet characters, adelizations, and  $L$-functions}\label{sec:Dirichlet}

A Dirichlet character of modulus $q \in \Z_{>0}$ is a  function $\xi:\Z\rightarrow\C$ satisfying the following properties:
\begin{enumerate}
  \item[1)] $\xi(nm)=\xi(n)\xi(m)$ for all $n,m\in\Z$ (``complete multiplicativity'');
  \item[2)] $\xi(n+q)=\xi(n)$ for all $n\in \Z$;
  \item[3)] and $\xi(n)=0$ if and only if $\gcd(n,q)>1$.
\end{enumerate}
Said differently, $\xi$ is the extension of a character of $(\Z/q\Z)^*$ to $\Z$ which takes the value zero on integers not coprime to $q$.   For example, the  Legendre (i.e., quadratic residue) symbol $n\mapsto (\frac{n}{p})$ is a Dirichlet character  of modulus $p$  for $p$ prime. The unique nontrivial Dirichlet character of modulus  4 is defined by
\begin{equation}\label{chi4def}
 n \  \ \longmapsto \  \  \left(\frac{-4}{n}\right) \  := \ \left\{
                              \begin{array}{ll}
                                1, & n\equiv1\ \ \pmod4 \\
                                -1, & n\equiv -1\pmod 4 \\
                                0, & n\ \text{even.}
                              \end{array}
                            \right.
\end{equation}
 Other examples of Dirichlet characters include
\begin{equation}\label{chi8def}
n \ \  \longmapsto \  \  \left(\frac{8}{n}\right) \  := \  \left\{
                              \begin{array}{ll}
                                1, & n\equiv\pm1\pmod8 \\
                                -1, & n\equiv \pm 3 \pmod 8 \\
                                0, & n\ \text{even,}
                              \end{array}
                            \right.
\end{equation}
and $n\mapsto (\frac{-8}{n}):=(\frac{-4}{n})(\frac{8}{n})$, both of which have modulus 8.  Together,
\begin{equation}\label{charset}
  {\mathcal C} \ \ := \ \ \left\{ \left(\frac{-4}{\cdot}\right),\left(\frac{8}{\cdot}\right),\left(\frac{-8}{\cdot}\right)
\right\} \ \cup \left\{ \left(\frac{\cdot}{p}\right) \,|\, p \ \text{an odd prime} \right\}
\end{equation}
  is a set of    order-two Dirichlet characters  of prime power modulus.  These characters, which have image   precisely  $\{-1,0,1\}$, are famously connected to the arithmetic of quadratic number fields and binary quadratic forms (see, e.g., \cite{davenport} for more background, as well as proofs of the analytic properties summarized below).

All Dirichlet characters can be viewed as characters on $\Q^*\backslash \A^*$, where $\A$ is the ring of adeles of $\Q$.  We shall now review this identification specialized to the  characters in $\mathcal C$ above, for which it simplifies.  Write each modulus  of a character $\xi\in {\mathcal C}$ as $q=p^m$ for some prime $p$ (e.g., $p=2$ for the first three characters listed in (\ref{charset})), and define quadratic characters $\xi_v$ of each completion $\Q_v^*$ as follows.
 \begin{itemize}
   \item  If $v=\infty$, then $\xi_v$ is either the trivial character (if $\xi(-1)=1$) or the $\operatorname{sgn}$ character $x\mapsto \frac{x}{|x|}$ (if $\xi(-1)=-1$) of $\Q_\infty^* = \R^*$.
   \item
 If $v=\ell<\infty$ is not equal to $p$, then $\xi_v$ is trivial on $\Z_\ell^*$ and takes the value $\xi(\ell)=\pm 1$ on $\ell$; together, this defines $\xi_v$ on $\Q_\ell^*=\{\ell^k|k\in\Z\} \times \Z^*_\ell$.
   \item  Finally, if $v=p$ then $\xi_p(p)=1$,  so that $\xi_p$ is determined by its restriction to $\Z_p^*$; $\xi_p$ is defined there to take the common value $\xi(a)$ on the subset $a+p^m \Z_p\subset \Z_p^*$, for $a$ coprime to $p$.
 \end{itemize}
The characters $\xi_v$ have a global adelic product $\xi_{\mathbb A}=\prod_v\xi_v$,  a character on $\A^*$ which is trivial on $\Q^*$.  Henceforth we will tacitly identify $\xi$ with $\xi_\A$.

The Dirichlet $L$-function of $\xi\in{\mathcal C}$ is defined as
\begin{equation}\label{Lpsidef}
\aligned
  L(s,\xi) \ \ & := \ \ \sum_{n\,=\,1}^\infty \xi(n)\,n^{-s}  \ \ = \ \ \prod_{\ell \text{~prime}} (1-\xi(\ell)\ell^{-s})^{-1} \, , \  \ \ \Re{s} \,>\,1\,, \\
\endaligned
\end{equation}
the latter being the Euler product for this Dirichlet series.    It has analytic properties similar to those of Riemann $\zeta$-function $\zeta(s)=\sum_{n=1}^\infty n^{-s}$, which we recall has a holomorphic continuation to $\C-\{1\}$, a simple pole at $s=1$, and obeys the functional equation $\zeta^*(s) = \zeta^*(1-s)$, where
\begin{equation}\label{zetastardef}
  \zeta^*(s) \ \ = \ \ \pi^{-s/2}\,\Gamma(\textstyle{\frac s2})\,\zeta(s)\,.
\end{equation}
The analogous statements for $L(s,\xi)$ involve  the  local factors
\begin{equation}\label{Lvdef}
  L(s,\xi_v) \ \ := \ \ \left\{
                             \begin{array}{ll}
                               (1-\xi(\ell)\ell^{-s})^{-1}, & v=\ell<\infty, \\
                               \pi^{-s/2}\,\Gamma(\frac s2), & v=\infty \text{~and~} \xi_\infty(-1)=1, \\
                               \pi^{-(s+1)/2}\,\Gamma(\frac{s+1}{2}), & v=\infty \text{~and~} \xi_\infty(-1)=-1,
                             \end{array}
                           \right.
\end{equation}
so that $L(s,\xi)=\prod_{v<\infty}L(s,\xi_v)$.
Then for $\xi \in {\mathcal C}$, $L^*(s,\xi):=L(s,\xi_\infty)L(s,\xi)$ has an entire continuation to $\C$ satisfying a functional equation of the form
\begin{equation}\label{LstarFE}
  L^*(1-s,\xi) \ \ = \ \  q^{s-1/2}\,L^*(s,\xi)\,.
\end{equation}
Similarly to (\ref{Lvdef}), the factor $q^{s-1/2}$ in the functional equation (\ref{LstarFE}) has local origins in terms of the   $\epsilon$-factors
\begin{equation}\label{epsilonvdef}
  \epsilon(s,\xi_v) \ \ := \ \ \left\{
                          \begin{array}{ll}
                            1, & v=\infty \text{~and~}\xi_\infty(-1)=1, \\
                            i, & v=\infty \text{~and~}\xi_\infty(-1)=-1, \\
                            q^{1/2-s}, & v=p \text{~and~}\xi_p(-1)=1, \\
                            -iq^{1/2-s}, & v=p \text{~and~}\xi_p(-1)=-1, \\
                            1, & \text{otherwise}.
                          \end{array}
                        \right.
\end{equation}
Indeed, $\prod_{v\le \infty} \epsilon(s,\xi_v)=q^{1/2-s}$.

Define the function $c(s,\xi)=\frac{L^*(s,\xi)}{q^{1/2-s}L^*(s+1,\xi)}=\frac{L^*(1-s,\xi)}{L^*(1+s,\xi)}$, which is itself the global product $c(s,\xi)=\prod_{v\le \infty}c(s,\xi_v)$ of local factors defined by
\begin{equation}\label{cvdef}
  c(s,\xi_v) \ \  = \ \ \frac{L(s,\xi_v)}{\epsilon(s,\xi_v)L(s+1,\xi_v)}\,,
\end{equation}
and clearly satisfies
\begin{equation}\label{cfes}
  c(s,\xi)\,c(-s,\xi) \ \ = \ \  1\,.
\end{equation}
Note that formulas (\ref{Lvdef})-(\ref{cfes}) also apply to the trivial character $\xi$ of modulus $q=1$, in which case $L(s,\xi)$ specializes to $\zeta(s)$, $L^*(s,\xi)$ specializes to $\zeta^*(s)$, and  $c(s,\xi)$ specializes to $c(s):=\frac{\zeta^*(s)}{\zeta^*(s+1)}=\frac{\zeta^*(1-s)}{\zeta^*(1+s)}$.

The following statement contains the analytic properties of $L(s,\xi)$ and $c(s,\xi)$ which will be used later in the paper:
\begin{lem}\label{lem:Lfunctionproperties}

1)
$c(0)=-1$,  $c'(-1)=-\frac \pi 6$, and $c(s)$ has a simple pole at $s=1$ with residue $\frac 6\pi$.  Furthermore, $c(s)$ is holomorphic and nonvanishing on $\R-\{-1,1\}$.

2) If $\xi\in {\mathcal C}$ then $c(s,\xi)$ is holomorphic and nonvanishing at all  $s\in\mathbb \Z$, with $c(0,\xi)=1$.  Furthermore, $c(s,\xi)\neq 0$ is holomorphic and nonvanishing for $|\Re{s}|>1$.

3) For $\Re{s}=0$, $|c(s)|=1$ and $|c(s,\xi)|=1$  for all $\xi\in\mathcal C$.
\end{lem}
\begin{proof}
 Statement 1) follows from the facts that $\zeta(s)$ is nonzero and holomorphic at all $s\neq 1$ on the real line, has a simple pole with residue 1 at $s=1$, and takes the value $\frac{\pi^2}{6}$ at $s=2$.

In contrast, $L^*(s,\xi)$ is entire for  $\xi\in {\mathcal C}$.   The convergent Euler product in (\ref{Lpsidef}) shows it is nonvanishing at $s\in \Z_{>1}$.  The  nonvanishing at $s=1$ is a consequence of (and  in fact was the impetus for) Dirichlet's famous class number formula.  The nonvanishing at $s=0,-1,-2,\ldots$ then follows from this using the functional equation (\ref{LstarFE}).  Statement 2) is a consequence of these facts and  the formula $c(s,\xi)=\frac{L^*(1-s,\xi)}{L^*(1+s,\xi)}$, along with the fact that  completed Dirichlet $L$-functions do not vanish outside of the critical strip.  Finally, for quadratic characters $\xi$ we have $L^*(s,\xi)\in \R$ for $s\in\R$, and the Schwartz reflection principle applied to this last formula shows statement 3) for $\xi\in\mathcal C$; the analogous statement for $c(s)=\frac{\zeta^*(1-s)}{\zeta^*(1+s)}$ is derived the same way.
\end{proof}

Let $\chi=\prod_v \chi_v$ be a quadratic character of $T(\Q)\backslash T(\A)$,
which we assume has the form $\chi=\xi\circ\delta$ from (\ref{globalcompwithxi}) for some $\xi\in\mathcal C$ and   $\delta\in X(T)=\Lambda_{\on{wt}}$ (recalling that $G$ is simply connected).
Recall from (\ref{coroothsnotation1}) that the coroot $\alpha^\vee$ of any root $\alpha\in \Delta$ can be identified with an algebraic morphism from $GL(1)$ to $T$, so that the composition $\chi\circ \alpha^\vee$ is a quadratic character of $\Q^*\backslash \A^*$ satisfying
\begin{equation}\label{chineginverse}
    \chi\circ((-\alpha)^\vee) \ \ = \ \ (\chi\circ\alpha^\vee)^{-1}\ \ = \ \ (\chi\circ\alpha^\vee)\,.
\end{equation}
Likewise, the local compositions $\chi_v\circ \alpha^\vee$ are quadratic characters of $\Q_v^*$ and $\chi\circ\alpha^\vee=\prod_v (\chi_v\circ\alpha^\vee)$, which is either trivial or   the Dirichlet character $\xi$ depending on whether the algebraic map $\delta\circ \alpha^\vee:\mathbb{G}_m\rightarrow \mathbb{G}_m$ is an even or odd power monomial.
Define the local factors
\begin{equation}\label{cvwdef}
  c_v(w, \lambda,\chi_v)  \ \ = \ \  \prod_{\alpha> 0, \; w\alpha < 0 }
 c_v( \la \lambda, \alpha^\vee \ra, \chi_v\circ \alpha^\vee)\,, \ \ v\,\le\,\infty\,,
\end{equation}
as well as their  global product
\begin{equation}\label{cwdef}
\aligned
 c(w,\lambda,\chi)  \ \ & = \ \  \prod_{v\le \infty} c_v(w, \lambda,\chi_v) \ \  = \  \  \prod_{\alpha> 0, \; w\alpha < 0 }   c( \la \lambda, \alpha^\vee \ra, \chi \circ \alpha^\vee)    \ , \\
\endaligned
\end{equation}
where $w\in W$
 and the notation  $\alpha>0$ (resp., $\alpha<0$) is shorthand for $\alpha\in \Delta_+$ (resp., $\alpha\in \Delta_{-}$).
  The Weyl group acts on $\chi$ by the formula
  \begin{equation}\label{wonchi}
    (w\chi)(t) \ \ := \ \ \chi(w^{-1}t{w})\,,
  \end{equation}
  which is consistent with (\ref{wtwvarpi}); in particular, since $\chi=\xi\circ\delta$ one has
  \begin{equation}\label{wchiwdelta}
    w\chi \ \ = \ \ \xi\circ w\delta\,.
  \end{equation}
  The following result is a fairly standard consequence of (\ref{cfes})  and (\ref{cwdef}).

\begin{prop}\label{prop:cw1w2}  Let $\chi$ be a quadratic character of $T(\Q)\backslash T( \A)$ of the form $\chi\circ \delta$, with $\chi\in\mathcal C$ and $\delta\in X(T)=\Lambda_{\on{wt}}$.
Then
\begin{equation}\label{cw1w2}
  c(w_1w_2,\lambda,\chi) \ \ = \ \ c(w_1,w_2\lambda,w_2\chi)\,c(w_2,\lambda,\chi)
\end{equation}
for any $w_1,w_2\in W$.
\end{prop}
\begin{proof}
We start by partitioning the set of roots in the product definition of $c(w_1w_2,\lambda,\chi)$ from (\ref{cwdef}) as
\begin{equation}\label{cw1w2pf1}
    S \ \ =\ \ \{ \alpha> 0 \, | \, w_1w_2\alpha < 0\} \ \  = \ \ S_1\,\sqcup\,S_2\,,
\end{equation}
where
\begin{equation}\label{cw1w2pf2}
S_1 \ \ = \ \ \{ \alpha> 0  \, | \,  w_2 \alpha < 0 \ \text{and} \  w_1w_2\alpha < 0\}
\end{equation}
and
\begin{equation}\label{cw1w2pf3}
S_2 \ \ = \ \ \{ \alpha> 0 \,  |  \, w_2 \alpha > 0\ \text{and} \  w_1w_2\alpha < 0\}\,.
\end{equation}
Then
\begin{equation}\label{cw1w2pf4}
S_1 \ \ \subset \ \  S_3 \ \ = \ \ \{ \alpha> 0 \,|\, w_2 \alpha < 0\}\,,
\end{equation}
the roots in the product definition of $c(w_2,\lambda,\chi)$.  Also,
$S_2$ is in bijective correspondence with
\begin{multline}\label{cw1w2pf5}
S_4 \ = \ \{ \gamma > 0 \,|\, w_2^{-1} \gamma > 0\ \text{and} \  w_1\gamma < 0\} \ \ \subset \ \
S_5 \ =  \
\{ \gamma > 0 \,|\,  w_1\gamma < 0\}
\end{multline}
via $\gamma = w_2 \alpha$; $S_5$ is the set of roots in the product definition of $c(w_1,w_2\lambda,w_2\chi)$.

Therefore,
\begin{multline}\label{cw1w2pf6}
c(w_1w_2,\lambda,\chi) \ \ = \ \
\prod_{\alpha \in S}c( \la \lambda, \alpha^\vee \ra, \chi \circ \alpha^\vee)
 \ \ = \ \ \prod_{\alpha \in S_1}c( \la \lambda, \alpha^\vee \ra, \chi \circ \alpha^\vee)
\prod_{\alpha \in S_2}c( \la \lambda, \alpha^\vee \ra, \chi \circ \alpha^\vee) \\
= \  \ \prod_{\alpha \in S_1}c( \la \lambda, \alpha^\vee \ra, \chi \circ \alpha^\vee)
\prod_{\gamma \in S_4}c( \la w_2\lambda, \gamma^\vee \ra, w_2\chi \circ \gamma^\vee)\,,
\end{multline}
where in the last product we have  used the Weyl-invariance of $\langle \cdot,\cdot \rangle$ along with the cocharacter identities
$(w_2^{-1}   \gamma)^\vee(t) = w_2^{-1} \gamma(t) w_2$   and $(w_2\chi\circ \gamma^\vee)(s) = (w_2\chi)(\gamma^\vee(s))=\chi(w_2^{-1}\gamma^\vee(s)w_2)$ from  (\ref{wvarpi}).

In terms of the above sets,
\begin{equation}\label{cw1w2pf7}
\aligned
    \frac{c(w_1,w_2\lambda,w_2\chi)\,c(w_2,\lambda,\chi)}{
c(w_1w_2,\lambda,\chi)} \ \ & = \ \ \prod_{\gamma \in S_5-S_4} c( \la w_2\lambda, \gamma^\vee \ra, w_2\chi \circ \gamma^\vee)
\prod_{\alpha \in S_3-S_1} c( \la \lambda, \alpha^\vee \ra, \chi \circ \alpha^\vee)
\\ &  = \ \
\prod_{\srel{\alpha < 0, w_2\alpha >  0 }{ w_1w_2\alpha < 0}}  c( \la \lambda, \alpha^\vee \ra, \chi \circ \alpha^\vee)\,
\prod_{\srel{\alpha > 0, w_2 \alpha < 0 }{ w_1 w_2 \alpha > 0} }  c( \la \lambda, \alpha^\vee \ra, \chi \circ \alpha^\vee)\,.
\endaligned
\end{equation}
  The result now follows since the roots appearing in the second product
are precisely the negatives of the ones appearing in the  first product, and since
\begin{equation}\label{cw1w2pf8}
c( \la \lam, \alpha^\vee \ra, \chi\circ \alpha^\vee)\,
c( \la \lam, -\alpha^\vee \ra, \chi\circ (-\alpha^\vee))
\ \ = \ \ 1
\end{equation}
by \eqref{cfes} and \eqref{chineginverse}.
\end{proof}

\section{Borel Eisenstein series}\label{sec:BorelEisensteinSeries}

In this section we review some material from the theory
of Eisenstein series attached to the Borel subgroup $B$; see, for example,   \cite{Kim-borel,MW} for further reference.  Let $\chi=\prod_v\chi_v$ be a global quadratic character of $T(\Q)\backslash T(\A)$
and let  $f(g,\lambda) \in I(\lambda,\chi)$ be a   global flat section.  (The second argument of $f$ will sometimes be omitted when no confusion arises.)
The {\it minimal parabolic Eisenstein series} is defined as the sum
\begin{equation}\label{boreleisdef}
E(f, \lambda, g)
\ \ = \ \  \sum_{\gm \,\in\, B(\Q) \bs G(\Q)}
f(\gm g, \lambda)\,, \  \ \ g\,\in\,G(\A)\,,
\end{equation}
which is convergent in the Godement range $\{\lambda\in {\frak a}_\C^*| \Re\langle \lam-\rho,\alpha^\vee\rangle > 0$ for all $ \alpha\in \Sigma\}$.
It is a theorem of Langlands that $E(f, \lambda, g)$ has a meromorphic continuation to all $\lambda \in {\frak a}_{\C}^*$.

Using the Bruhat decomposition for $G(\Q)$, (\ref{boreleisdef}) can be rewritten as
\begin{equation}\label{rewritteneisdef}
  E(f, \lambda, g)
 \ \ = \ \ \sum_{w\,\in\,W} \sum_{\gm \in B(\Q) \bs (B(\Q)w^{-1}\!B(\Q))}
f(\gm g, \lambda)\,.
\end{equation}
Here the second sum is independent of the choice of representative of the Weyl group element $w^{-1}$.
Coset representatives for $B(\Q)\backslash (B(\Q)w^{-1}B(\Q))$ are provided by $w^{-1}(N(\Q)\cap w N_{-}(\Q)w^{-1})$.  Langlands utilized them after rearranging  an integration over $N$ to derive his constant term formula
\begin{equation}\label{eq: constantterm}
\int_{N(\Q)\backslash N(\A)} E(f, \lambda, ug)\, du \ \
=  \ \ \sum_{w\in W}\,
[M(w, \lambda,\chi)f](g,\lam)\,,
\end{equation}
where
for each $w\in W$    the intertwining operator
\begin{equation}\label{Mwlchisends}
 M(w,\lambda, \chi)\, : \, I(\lambda,\chi)  \ \ \longrightarrow \ \
I(w\lambda,w\chi)
 \end{equation}
is defined  by
\begin{equation}\label{Mdef}
[ M(w, \lambda,\chi)f](g)  \ \ = \ \  \int_{N(\A) \,\cap\, w N_-(\A) w^{-1}}
f(w^{-1} ug) \, du\,,
 \end{equation}
initially as a convergent integral for $\lambda\in \frak a_\C^*$
satisfying
\begin{equation}\label{Mdefrange}
\Re \la \lambda, \alpha^\vee \ra   \ \ > \ \  \la \rho, \alpha^\vee\ra \, , \quad
 \forall \alpha > 0\text{~such that~} w\alpha < 0\,,
\end{equation}
and then by meromorphic
 continuation to all $\lambda\in {\frak a}_\C^*$
 (see \cite[Prop.~II.1.6(iv)]{MW}).
 The analytic properties of $M(w,\lambda,\chi)$ are crucial to the rest of our analysis, and are the subject of the next section.

 \section{Properties of intertwining operators}\label{sec:intertwine}

Let $w\in W$ be an element of length $\ell=\ell(w)$; that is, $w$ is a reduced product $w=w_{\beta_1}\cdots w_{\beta_\ell}$ of $\ell$ simple reflections   attached to positive simple roots $\beta_1,\ldots, \beta_\ell$.
It follows from factoring the range of integration in definition (\ref{Mdef}) that
 the intertwining operators satisfy the  composition identity
 \begin{equation}\label{langlandslemma}
 M(w_1w_2,\lambda, \chi)  \ \ = \ \  M(w_1,w_2\lambda, w_2\chi)\circ M(w_2, \lambda,  \chi)
 \end{equation}
whenever $\ell(w_1w_2)=\ell(w_1)+\ell(w_2)$, an identity which in fact holds for all $w_1,w_2\in W$ \cite[Theorem IV.1.10(b)]{MW}.
In particular,   $M(w,\lambda,\chi)$ is the composition
\begin{multline}\label{Mwlengthell}
    M(w,\lambda,\chi) \ \ = \ \ M(w_{\beta_1},w_{\beta_2}\cdots w_{\beta_\ell}\lambda,w_{\beta_2}\cdots w_{\beta_\ell}\chi)\,\circ \\
    M(w_{\beta_2},w_{\beta_3}\cdots w_{\beta_\ell}\lambda,w_{\beta_3}\cdots w_{\beta_\ell}\chi)\circ
    \cdots
    \circ M(w_{\beta_{\ell}},\lambda,\chi)\,,
\end{multline}
in which each  factor  is defined by (\ref{Mdef}) as the integration
\begin{equation}\label{Mwsimple}
    [M(w_{\beta},\lambda,\chi)f](g) \ \ = \ \ \int_{\A} f(w_{\beta}^{-1}u_\beta(x)g)\,dx
\end{equation}
over the one-parameter unipotent subgroup $u_\beta(\cdot)$ attached to $\beta$
from Section~\ref{sec:chevalleygroups}.
The     operators $M(w_{\beta},\lambda,\chi)$ are essentially $SL_2$ intertwining operators, whose analytic properties can be directly understood using  calculus.

The intertwining operators $M(w,\lambda,\chi)$ are global tensor products of their local analogs
\begin{equation}\label{Mvmap}
    M_v(w,\lambda, \chi_v): I_v( \lambda,\chi_v)  \ \ \longrightarrow \ \
I_v(w\lambda,w\chi_v)\,,
\end{equation}
defined on $f_v\in I_v(\lambda,\chi_v)$ by
\begin{equation}\label{Mvdef}
[ M_v(w, \lambda,\chi_v)f_v](g) \ \  = \ \ \int_{N(\Q_v) \cap w \! N_-(\Q_v) w^{-1}}
f_v(w^{-1} ug) \, du
 \end{equation}
as an absolutely convergent integral in the region $\{\Re\langle \lambda,\alpha^\vee\rangle > 0 | \alpha >0,w\alpha<0\}$; (\ref{Mvdef})  then   meromorphically continues to $\lambda\in {\frak a}^*_\C$
(see \cite[Proposition 7.8]{Knapp} and \cite[Theorem 2.2]{Schiffmann} for $v= \infty$,
and  \cite[Theorem 5.3.5.4]{Silberger} and \cite[Remark 6.4.5]{Casselman}
for $v < \infty$).  Since elements $I_v(\lambda,\chi_v)$ are assumed to be $K_v$-finite by definition, the action of  $M_v(w,\lambda,\chi_v)$   can   be described in terms of  a sequence of finite-dimensional matrices whose entries are meromorphic functions of $\lambda$.

\textcolor{blue}{Note added after publication:~unfortunately, a misconception in earlier references about intertwining operators has bled through here.  We take the opportunity to point this out, what its impact is on statements here, and finally why it has no bearing on the final results.  These comments will all be in blue, so that the original text is left undisturbed wherever possible.  While the integrand in (\ref{Mdef}) does depend on the choice of $w\in G(\Z)$ from Section~\ref{sec:chevalleygroups}, the {\it global} intertwining operators are in fact independent of this choice.  However, some earlier statements in the literature notwithstanding, the same cannot be said for their {\it local} analogs $M_v(w,\cdot,\cdot)$ (\ref{Mvdef}) and $R_v(w,\cdot,\cdot)$ (\ref{Rvintop}), which do depend on the choice of $w\in G(\Z)\cap N(T)$.  Various choices of this representative differ by multiplication by an element in $T\cap G(\Z)$, and merely change the value of the intertwining operators by a scalar multiple via (\ref{Inuspace}).  There is in general no precise compatibility between these various choices in the local setting, due to (and only due to) this scalar ambiguity.
}

There is a well-known normalization of local intertwining operators  $M_v(w,\lambda,\chi_v)$ suggested by Langlands (and motivated by  the functional equations of $L$-functions and Eisenstein series) in \cite[Appendix II]{L}.  Let
\begin{equation}\label{Rvintop}
\aligned
 R_v(w,\lambda,\chi_{v}) \ \ & :=  \ \ c_v(w,\lambda,\chi_v)^{-1} M_v(w,\lambda,\chi_{v}) \\
    &  = \ \ \(\prod_{\alpha>0,w\alpha<0} \epsilon(\langle \lambda,\alpha^\vee\rangle,\chi_v\circ\alpha^\vee)\, \frac{L(\langle \lambda,\alpha^\vee\rangle+1,\chi_v\circ\alpha^\vee)}{L(\langle \lambda,\alpha^\vee\rangle,\chi_v\circ\alpha^\vee)}\)\!M_v(w,\lambda,\chi_{v})\,,
\endaligned
\end{equation}
where we have used (\ref{cvdef}) and (\ref{cvwdef}).
This normalization is chosen so that
\begin{equation}\label{Rvonspherical}
    R_v(w,\lambda,\chi_{v})f^\circ_{v,\lambda,\chi_v} \ \ = \ \
    f^\circ_{v,\lambda,\chi_v}
\end{equation}
when $\chi_v$ is unramified,
where $f^\circ_{v,\lambda,\chi_v}$ denotes the spherical flat section whose restriction to $K_v$ is identically one.
Define the global operator $R$ as the tensor product of all $R_v$, i.e.,
\begin{equation}\label{MRcidentity}
    M(w,\lambda,\chi) \ \ =  \ \ c(w,\lambda,\chi)\,R(w,\lambda,\chi)
\end{equation}
(see (\ref{cwdef})).
If $f=\otimes_v f_v$ is a pure tensor  with $f_v=f^\circ_{v,\lambda,\chi_v}$ for all $v$ outside a finite set of places $S$ containing the ramified places for $\chi$, then
\begin{equation}\label{MandRontensors}
    \aligned
    M(w,\lambda,\chi)f \ \ & = \ \ \otimes M_v(w,\lambda,\chi_v)f_v \\
   \text{and} \ \ \ \ \ \ R(w,\lambda,\chi)f \ \ & = \ \ \otimes R_v(w,\lambda,\chi_v)f_v \ \  = \ \
    \(\otimes_{v\in S} R_v(w,\lambda,\chi_v)f_v\) \otimes \(\otimes_{v\notin S}f^\circ_{v,\lambda,\chi_v}\).
    \endaligned
\end{equation}
In particular, the action of $R(w,\lambda,\chi)$ on pure tensors in $I(\lambda,\chi)$ involves only finitely many factors in a nontrivial way, so $R(w,\lambda,\chi)$ and  $M(w,\lambda,\chi)$ act as   finite-dimensional matrices on any particular element of $I(\lambda,\chi)$  (see the comments at the end of Section~\ref{sec:induced}).

The following theorem collects a number of   analytic properties of intertwining operators that will be used in the last two sections of this paper.  Some of these are well-known (even in more general contexts, e.g., \cite[\S2]{Shahidi81}), but others are not or are scattered in the literature.  See also \cite{Kim-G2,Kim-borel,Kim-Shahidi,Yuanli} for further background.

\begin{thm}\label{newintertwintheorem} Let $G$ be a simply-connected Chevalley group and
let $\chi=\prod_v\chi_v$ be a quadratic character of $T(\Q)\backslash T(\A)$ of the form $\xi\circ \delta$ from (\ref{globalcompwithxi}), where $\xi$ is one of the characters in the set $\mathcal C$ defined in (\ref{charset}) and $\delta\in X(T)=\Lambda_{\on{wt}}$.   Then the following properties hold.

\begin{enumerate}[label=\textnormal{\arabic*)}]

\item\label{int:1} For each
$\lambda$ in its range of absolute  convergence $\{\Re\langle \lambda,\alpha^\vee\rangle > 0 | \alpha >0,w\alpha<0\}$, the local intertwining operator $M_v(w,\lambda,\chi_v)$ from (\ref{Mvdef}) is not identically 0.

\item\label{int:2}  The intertwining operators $M(w,\lambda,\chi)$, $M_v(w,\lambda,\chi_v)$, $R(w,\lambda,\chi)$, and $R_v(w,\lambda,\chi_v)$   all have meromorphic continuations (as operators) to $\lambda \in {\frak a}_\C^*$ (that is, the evaluation   of their action on any fixed flat section, at any fixed group element, is meromorphic in $\lambda$).

\item\label{int:3}    Formula (\ref{langlandslemma}) holds for any $w_1, w_2\in W$.

\item\label{int:4}  For any $w_1, w_2\in W$, $R_v(w_1w_2,\lambda,\chi_v)=R_v(w_1,w_2\lambda,w_2\chi_v)\circ R_v(w_2,\lambda,\chi_v)$.  \textcolor{blue}{Note:~this is correct as stated, as long as $w_1$, $w_2$, and $w_1w_2$ are interpreted  as an elements of $G(\Z)\cap N(T)$.  If, however, they are thought of as  Weyl group elements, the formula is only correct up to a nonzero scalar.}

\item\label{int:5}   Let $\beta\in\Sigma$ and $w_\beta\in W$ be its associated simple  reflection.
Then $R_v(w_\beta,w_\beta\lambda,w_\beta\chi)\circ R_v(w_\beta,\lambda,\chi)$ \textcolor{blue}{is a nonzero scalar multiple of the identity operator} and
in particular
\begin{equation}\label{Mvintopsimplereflect}\begin{aligned}
  &M_v(w_\beta,w_\beta\lambda,w_\beta\chi)\circ M_v(w_\beta,\lambda,\chi)\equiv \\
  &\frac{ L(\langle \lambda,\beta^\vee\rangle,\chi_v\circ\beta^\vee)L(-\langle \lambda,\beta^\vee\rangle,\chi_v\circ\beta^\vee)}{\epsilon(\langle \lambda,\beta^\vee\rangle,\chi_v\circ\beta^\vee)\epsilon(-\langle \lambda,\beta^\vee\rangle,\chi_v\circ\beta^\vee)L(\langle \lambda,\beta^\vee\rangle+1,\chi_v\circ\beta^\vee)L(\langle -\lambda,\beta^\vee\rangle+1,\chi_v\circ\beta^\vee)}
\end{aligned}
\end{equation}
\textcolor{blue}{up to scalar multiples}, recalling that $\chi_v$ is a quadratic character.

\item\label{int:6}  If $\beta$ is a positive root such that $\langle \lambda,\beta^\vee\rangle=-1$ and $\chi_v\circ \beta^\vee$ is trivial,   then the right-hand side of (\ref{Mvintopsimplereflect}) vanishes.

\item\label{int:7}   $\(\prod_{\alpha>0,w\alpha<0}L_v(\langle \lambda,\alpha^\vee\rangle,\chi_v\circ \alpha^\vee)\)^{-1}M_v(w,\lambda,\chi_v)$ is entire.

\item\label{int:8} $R(w,\lambda,\chi)$ as well as each $R_v(w,\lambda,\chi_v)$ is holomorphic in $\{\Re \langle \lambda,\alpha^\vee\rangle > -1, \forall \alpha>0\text{~with~}w\alpha<0\}$.

\item\label{int:9}  $R_v(w,\lambda,\chi_v)$ is not the zero operator when $\langle \lambda,\alpha^\vee\rangle > 0 , \forall \alpha>0\text{~with~}w\alpha<0$.

\item\label{int:10}  $R_v(w,\lambda,\chi_v)$ is an isomorphism when $\langle \lambda,\alpha^\vee\rangle = 0, \forall \alpha>0\text{~with~}w\alpha<0$.

\item\label{int:11} $R_v(w,\lambda,\chi_v)$ is not the zero operator if $\Re\lambda$ is dominant.

\item\label{int:12} $R(w,\lambda,\chi)$ is not the zero operator  if $\Re\lambda$ is dominant.

\item\label{int:13} If $w$ can be written as a reduced word $w_{\beta_1}\cdots w_{\beta_\ell}$ such that each
$\langle \lambda,\beta_i^\vee\rangle = 0$ and $\chi\circ \beta_i^\vee$ is trivial, then $M(w, \lambda, \chi)$ is the scalar operator
$(-1)^{\ell}.$

\item\label{int:14} If $\Re{\lambda}$ is dominant, then there exists a vector $f_v\in I_v(\lambda,\chi_v)$ which is not annihilated by any of the operators $R_v(w,\lambda,\chi_v)$, $w\in W$.  Furthermore, there exists a pure tensor $f$ not annihilated by any of the operators
 $R(w,\lambda,\chi)$, $w\in W$. (Note that by property \ref{int:8}, these intertwining operators are holomorphic at $\lambda$.)
\end{enumerate}
\end{thm}
\begin{proof}

\ref{int:1} At $g=w$, formula  (\ref{Mvdef}) can be written as
\begin{equation}\label{intoveru-}
  [M_v(w,\lambda,\chi_v)f_v](w) \ \ = \ \ \int_{U_{-}(\Q_v)}f_v(u_{-})\,du_{-}\,,
\end{equation}
with $u_{-}\in U_{-}=w^{-1}Nw\cap N_{-}$. Using the Iwasawa decomposition  write $u_{-}=b(u_{-})k(u_{-})$, with $b(u_{-})\in B(\Q_v)$ and $k(u_{-})\in K_v$,   so that the integrand is
\begin{equation}\label{fu-}
  f_v(u_{-}) \ \ = \ \  |b(u_{-})|_{v}^{\lambda+\rho}\,\left[\chi_v(b(u_{-})) \,f_v(k(u_{-}))\right].
\end{equation}
 Recall from (\ref{KTcompatlocal}) that this  expression is well-defined independently of the choice of Iwasawa decomposition, since $k(u_{-})$ is determined   up to left-multiplication by elements of $B(\Q_v)\cap K_v$.

  At $u_{-}=e$ both $f_v(u_{-})$ and the bracketed expression in (\ref{fu-}) take  the value $f_v(e)$.
 Writing $k(u_{-})=b(u_{-})^{-1}u_{-}$ shows that $u_{-}$ is the $N_{-}(\Q_v)$-factor of the
 ``LU'' (or more accurately, ``UL'') decomposition
  of $k(u_{-})$; in particular, such decompositions are given by algebraic formulas which guarantee that
  the bijective correspondence $u_{-}\leftrightarrow k(u_{-})$ between $U_{-}(\Q_v)$ and its $k(\cdot)$-image in $(B(\Q_v)\cap K_v)\backslash K_v$   is   continuous in both directions.  Furthermore, for any open neighborhood $O_B$ of the identity in $B(\Q_v)$ satisfying $O_B=O_B\cdot (B(\Q_v)\cap K_v)$, there exists some open neighborhood $O_K=(B(\Q_v)\cap K_v)O_K$ of the identity in $(B(\Q_v)\cap K_v)\backslash K_v$  such that $k(u_{-})\in O_K\Longrightarrow b(u_{-})\in O_B$.

Suppose first that $v<\infty$.  Take $O_B$ sufficiently small so that $b^\alpha \in \Z_v^*$ for each $\alpha\in \Sigma$ and $b\in O_B$, and consider local flat sections $f_v$ whose restriction to $K_v$ is  supported in the set $O_K$ above.  Then the factor $|b(u_{-})|_v^{\lambda+\rho}$ in (\ref{fu-}) equals 1 on the range of support of $f_v(u_{-})$ and $[M_v(w,\lambda,\chi_v)f_v](w)$ is independent of $\lambda$.  However, the integral (\ref{intoveru-}) cannot vanish for all $\lambda$ and  all $f_v$ with such small support (as can be seen, for example, by shrinking the support to a point or using a positivity argument).

For $v=\infty$ this argument requires technical modifications.  The intertwining operators $M_v(w,\lambda,\chi_v)$ extend continuously   to the   completion $I^\infty(\lambda,\chi_v)$ of $I(\lambda,\chi_v)$ consisting of all smooth functions satisfying (\ref{KTcompatlocal}) \cite{vogan-wallach,Wallach} -- and even furthermore to its distributional completion \cite{Casselman-smooth}.  Were   $M_v(w,\lambda,\chi_v)$ identically zero on $I(\lambda,\chi_v)$,  it would also vanish identically on $I^\infty(\lambda,\chi_v)$ and hence on its distributional completion.  However, (\ref{intoveru-}) evaluates to 1 when $f_v$ is replaced by a delta function at $e$.

\ref{int:2} For the global operators $M(w,\lambda,\chi)$, see   \cite[Prop II.1.6(iv)]{MW}; for the local operators $M_v(w,\lambda,\chi_v)$, see
\cite[Theorem 2.2]{Schiffmann} (in the archimedean
case) and \cite[Remark 6.4.5]{Casselman} (in the nonarchimedean
case).
   Since $c(w,\lambda,\chi)$ and $c(w,\lambda,\chi_v)$ are meromorphic by Lemma~\ref{lem:Lfunctionproperties} and (\ref{cvwdef})-(\ref{cwdef}), the assertions for $R(w,\lambda,\chi)$ and $R(w,\lambda,\chi_v)$ follow from those for $M(w,\lambda,\chi)$ and $M_v(w,\lambda,\chi_v)$.

\ref{int:3} See \cite[Theorem IV.1.10(b)]{MW}.

\ref{int:4} This is a result of Arthur \cite[property ($R_2$), p.~28] {Arthur-IntOpI}) (see also earlier work of Shahidi \cite[(3.4)]{Shahidi-local}).

\ref{int:5} Since  $R_v(e,\lambda,\chi_v)$ is the trivial operator, this follows from part~\ref{int:4}, (\ref{Rvintop}), and (\ref{chineginverse}).

\ref{int:6} This follows from (\ref{Lvdef}), (\ref{epsilonvdef}), and part 5).

\ref{int:7} This is a result of \cite{Winarsky} (in the nonarchimedean case) and \cite[p.~110]{Shahidi} (in the archimedean case).

\ref{int:8} Consider (\ref{Rvintop}).   The factor $\epsilon(\langle \lambda,\alpha^\vee\rangle,\chi_v\circ\alpha^\vee)$ is entire and nonvanishing (see (\ref{epsilonvdef})), while the numerator $L(\langle \lambda,\alpha^\vee\rangle+1,\chi_v\circ\alpha^\vee)$ is holomorphic for $\Re\langle \lambda,\alpha^\vee\rangle > -1$ (see (\ref{Lvdef})).  The holomorphy of $R_v$ then follows from part 7).  The global statement follows from   (\ref{MandRontensors}), in which $R_v(w,\lambda,\chi_v)$ acts nontrivially for only finitely many $v$.

\ref{int:9} This follows from property~\ref{int:1} and the fact $c_v(w,\lambda,\chi_v)$ is nonzero and holomorphic in this range (cf. \cite[p.~29, property $(R_7)$]{Arthur-IntOpI}).

\ref{int:10} First consider $w=w_\alpha$, $\alpha \in\Sigma$.  Then by part~\ref{int:8} $R_v(w_\alpha,\lambda,\chi_v)$ is holomorphic at $\lambda$ for which $\langle \lambda,\alpha^\vee\rangle =0$.  Since $\langle w_\alpha\lambda,\alpha^\vee\rangle = - \langle \lambda,\alpha^\vee\rangle$ and $w_\alpha\chi_v=\chi_v$, both $R_v(w_\alpha,\lambda,\chi_v)$ and $R_v(w_\alpha,w_\alpha\lambda,w_\alpha\chi_v)$ are holomorphic at these $\lambda$.  According to property \ref{int:4} (e.g., as it is used in proving property \ref{int:5}), the composition of these operators is \textcolor{blue}{a scalar multiple of} the identity.  Thus each is an isomorphism.

In general, if   $w$ is written as the reduced word $w=w_{\beta_1}\cdots w_{\beta_\ell}$,  $\beta_1,\ldots,\beta_\ell\in\Sigma$, property~\ref{int:4} implies the factorization
\begin{multline}\label{Rvfactorintosimple}
   R_v(w,\lambda,\chi_v) \ \ = \ \ R_v(w_{\beta_1},w_{\beta_2}\cdots w_{\beta_\ell}\lambda,w_{\beta_2}\cdots w_{\beta_\ell}\chi_v)\circ \\
    R_v(w_{\beta_2},w_{\beta_3}\cdots w_{\beta_\ell}\lambda,w_{\beta_3}\cdots w_{\beta_\ell}\chi_v)\circ
    \cdots
    \circ R_v(w_{\beta_{\ell}},\lambda,\chi_v)\,,
\end{multline}
\textcolor{blue}{holds up to a scalar multiple}.
Each of these factors is an isomorphism, since $w_{\beta_j}$ is a simple reflection and $\langle w_{\beta_{j+1}}\cdots w_{\beta_\ell}\lambda,\beta_j^\vee\rangle=
\langle \lambda,w_{\beta_\ell}\cdots w_{\beta_{j+1}}\beta_j^\vee\rangle=0$ (using the assumption together with the standard fact that $w_{\beta_\ell}\cdots w_{\beta_{j+1}}\beta_j$ is a positive root whose sign is flipped by $w$).

\ref{int:11}
 Let $\Delta_M=\{\alpha \in \Delta| \Re \langle \lambda,\alpha^\vee \rangle  =0 \}.$
 It is the root system of a standard Levi subgroup.
 Let
$W_M\subset W$ be the corresponding Weyl group,
and  factor     $w$ as $w=w_1w_2$, where   $w_1\alpha>0$ for each $\alpha\in \Delta_M$, and $w_2\in W_M$.
 Property~\ref{int:10} implies that $R_v(w_2,\lambda,\chi_v)$ is an isomorphism.

The rest follows from properties~\ref{int:4} and \ref{int:9} once
we show that property~\ref{int:9} applies to $R_v(w_1,w_2\lambda,w_2\chi_v),$ i.e., that
$\Re\la w_2\lambda, \alpha^\vee \ra > 0$
for all $\alpha > 0$ with $w_1 \alpha < 0.$
Take such a root $\alpha$ and let $\beta = w_2^{-1} \alpha.$
By definition of $w_1, \alpha \notin \Delta_M.$
Hence $\beta \notin \Delta_M$, and in fact $\beta>0$ since $W_M$ does not flip the sign of any roots outside of $\Delta_M$.
In particular, $\Re\la \lambda, \beta^\vee \ra > 0.$
Applying $w_2$ on both sides of the invariant
pairing $\la\ , \ \ra$ gives $\Re\la w_2\lambda, \alpha^\vee \ra > 0.$

\ref{int:12} This follows from (\ref{MandRontensors}) and property~\ref{int:11}.

\ref{int:13}
The hypotheses imply that $w_{\beta_i}\lambda=\lambda$.
 Since $\chi(\beta_i^\vee(\cdot))=\xi( \beta_i^\vee(\cdot)^\delta)$ is assumed to be trivial but $\xi$ is nontrivial, it must be the case that $\beta_i^\vee(\cdot)^\delta$ is an even power.  According to (\ref{corootrootinteraction}),  $\langle \delta,\beta_i^\vee\rangle\in 2\Z$, and consequently $w_{\beta_i}\delta=\delta-\langle \delta,\beta_i^\vee \rangle \beta_i\equiv\delta\pmod{2\Lambda_{\on{wt}}}$.  It follows from the assumption $\chi=\xi\circ\delta$ that   $w_{\beta_i}\chi=\chi$, since the ratio of these last two characters is the square of a quadratic character.

   Each factor of $M(w,\lambda,\chi)$ in the composition formula \eqref{Mwlengthell} then has the form
 $M(w_{\beta_j}, \lambda,  \chi)$
  for some $j=1,\ldots,n$.    Thus the assertion reduces to the special case that $\ell=1$ and   $w$ is a simple reflection, say $w=w_\beta$ for some   $\beta\in\Sigma$.  Taking further into account (\ref{MRcidentity}) and the fact that
\begin{equation}\label{int13b}
 c(w_\beta,\lambda,\chi) \ \ = \ \ c(\langle \lambda,\beta^\vee\rangle,\chi\circ\beta^\vee) \ \  = \ \
 c(0,\chi\circ\beta^\vee) \ \ = \ \ -1
\end{equation}
by (\ref{cwdef}) and part 1) of Lemma~\ref{lem:Lfunctionproperties}, it suffices to show that $R(w_\beta,\lambda,\chi)\equiv 1$.  (Recall by property~\ref{int:8} that $\lambda$ is in the domain of holomorphy for $R(w_\beta,\cdot,\chi)$ and each $R_v(w_\beta,\cdot,\chi)$, $v\le \infty$.)  We will accordingly complete the proof by showing that
\begin{equation}\label{sufficestoshowthis}
 [R_v(w_\beta,\lambda,\chi_v)f_v](g) \ \ = \ \ f_v(g)\,,
\end{equation}
  for any $v\le\infty$ and $f_v(g)\in I(\lambda,\chi_v)$.  In light of the transformation law (\ref{Inuspace}) it suffices to verify this for $g\in K_v$ and, upon right translation, in fact merely at $g=e$.

Let $\Psi_\beta:SL_2 \rightarrow G$ denote the algebraic map sending $\ttwo 1x01\mapsto u_\beta(x)$ and $\ttwo 10x1\mapsto u_{-\beta}(x)$.
 According to (\ref{Mvdef}), if $f_v\in I(\mu,\chi_v)$ is a flat section one has
 \begin{equation}\label{int13bc}
   [M_v(w_\beta,\mu,\chi_v)f_v](e) \ \ = \ \ \int_{\Q_v} f_v\(\Psi_\beta\ttwo{0}{-1}{1}{x}  \)\,dx
 \end{equation}
 in the range of convergence  $\langle \mu,\beta^\vee \rangle >0$.  Let $n_xa_xk_x$ be an $SL(2,\Q_v)$-Iwasawa decomposition for $\ttwo{0}{-1}{1}{x}$, so that the integral is $\int_{\Q_v}\chi_v(\Psi_\beta(a_x))|\Psi_\beta(a_x)|_v^{\mu+\rho}f_v(\Psi_\beta(k_x))dx$ (see (\ref{lambdaCaction})).  The integrand simplifies to $|\Psi_\beta(a_x)|_v^{\mu+\rho}f_v(\Psi_\beta(k_x))$ since $\chi\circ \beta^\vee$ is assumed to be trivial.  The first factor comes from the embedded $SL_2$,  and can be written as $|\Psi_\beta(a_x)|_v^{\mu+\rho}=|\Psi_\beta(a_x)|_v^{(\langle \mu,\beta^\vee\rangle +1)\varpi_\beta}$,
 where $\varpi_\beta$ is the fundamental weight associated to $\beta$.  We write this expression using $SL_2$ matrices as $|d(a_x)|^{(\langle \mu,\beta^\vee\rangle +1)}_v$, where $a_x=\ttwo{d(a_x)}{0}{0}{d(a_x)^{-1}}$.

By assumption $\lambda$ lies in the hyperplane defined by  $\langle \mu,\beta^\vee \rangle =0$, which  is within the domain of holomorphy for the
   normalized intertwining operator $R_v(w_\beta,\mu,\chi_v)$ but on the boundary of the range of absolute convergence for the integral (\ref{int13bc}) defining $M_v(w_\beta,\mu,\chi_v)=c_v(\langle \mu,\beta^\vee\rangle,\chi_v\circ \beta^\vee)R_v(w_\beta,\mu,\chi_v)$.  We thus need to prove
 \begin{equation}\label{int13c}
   f_v(e) \ \ = \ \ \lim_{s\to 0} c_v(s,\chi_v\circ\beta^\vee)^{-1}\,\int_{\Q_v}
   |d(a_x)|^{s+1}f_v(\Psi_\beta(k_x))
   \,dx\,.
 \end{equation}
At this point, the assertion reduces to a calculation of a particular intertwining operator, namely one sending  the spherical principal series $I(0,\chi_{\on{triv},v})$ to itself  for the group $G(\Q_v)=SL(2,\Q_v)$, on the flat section $f_v^\beta(g)= f_v(\Psi_\beta(g))$.  It is a consequence of (\ref{Rvonspherical}) (which is itself the standard calculation of the ``Gindikin-Karpelevich integral'') that (\ref{int13c}) holds for  spherical vectors, i.e., when $f_v^\beta$ is constant on the maximal compact subgroup $K_v=SL(2,\Z_v)$ (for $v<\infty$) or $SO(2)$ (for $v=\infty$) of $SL(2,\Q_v)$.  Since $I(0,\chi_{\on{triv},v})$ is irreducible for $\chi_{\on{triv},v}$ trivial, the result holds on the full space.\footnote{Alternatively, one can argue directly that since $c_v(0,\chi_v\circ\beta^\vee)=0$, the limit on the right-hand side of (\ref{int13c}) is determined by the residue of the integral at $s=0$, whose value is itself influenced only by the values of $f_v(\Psi_\beta(k_x))$ for $|x|_v$ large -- that is, by $f(e)$.  The Gindikin-Karpelevich integral shows the residue has the desired value on spherical flat sections, and a smoothness argument shows that $\int_{\Q_v}|\Psi_\beta(a_x)|^{s+1}(f_v(\Psi_\beta(k_x))-f_v(e))dx$ is holomorphic at $s=0$.}

\ref{int:14} Let $w_{\text{long}}$ denote the long element of the Weyl group $W$.  By property~\ref{int:11}, $R_v(w_{\text{long}},\lambda,\chi_v)$   is not the zero operator, so there exists some vector $f_v\in I_v(\lambda,\chi_v)$ not in its kernel.
We shall now argue that  $R_v(w_{\text{long}},\lambda,\chi_v)$
factors through
$R_v(w,\lambda,\chi_v)$ for any $w\in W$.  Indeed, setting
  $w'=w_{\text{long}}w^{-1}$,   property~\ref{int:4} implies
\begin{equation}\label{prop5inprop14pf}
   0 \ \ \not\equiv \ \  R_v(w_{\text{long}},\lambda,\chi_v)f_v \ \ = \ \  R_v(w',w\lambda,w\chi_v)\left(R_v(w,\lambda,\chi_v)f_v\right).
\end{equation}
If $\beta$ is a positive root such that $w'\beta<0$, then $w^{-1}\beta=w_{\text{long}}^{-1}(w'\beta)$ is a positive root.  In particular, $\Re\langle w\lambda,\beta^\vee\rangle =  \Re\langle \lambda,w^{-1}\beta^\vee\rangle \ge 0$ by our assumption that $\Re{\lambda}$ is dominant.  Property~\ref{int:8} then asserts $R_v(w',\mu,w\chi_v)$ is holomorphic at $\mu=w\lambda$.  Thus $R(w,\lambda,\chi_v)f_v$ cannot vanish.  The global statement is then a consequence of (\ref{MandRontensors}).

\end{proof}

Recalling the isomorphism between $I(\lambda,\chi)$ and $\Ind_{B(\A) \cap K}^{K} \chi$  from the end of Section~\ref{sec:induced}, the
intertwining operators (\ref{Mwlchisends})
induce maps
\begin{equation}\label{smallmlamchimaps}
\mathbf{m}(w, \lambda,\chi): \Ind_{B(\A) \cap K}^K \chi  \ \ \longrightarrow \ \
\Ind_{B(\A) \cap K}^K w\chi
\end{equation}
by restriction of flat sections to $K$.  Similarly, we have   local analogs
\begin{equation}\label{smallmlamchimapsv}
\mathbf{m}_v(w, \lambda,\chi_v): \Ind_{B(\Q_v) \cap K_v}^{K_v} \chi_v  \ \ \longrightarrow \ \
\Ind_{B(\Q_v) \cap K_v}^{K_v} w\chi_v
\end{equation}
of (\ref{Mvdef}) induced by restriction to $K_v$.  One again has $\mathbf{m}(w,\lambda,\chi)=\otimes_{v\le\infty}\mathbf{m}_v(w,\lambda,\chi_v)$, and both operators act on any particular vector by finite-dimensional matrices whose entries are meromorphic functions of $\lambda$.  Similarly restricting $R(\cdot,\cdot,\cdot)$  and $R_v(\cdot,\cdot,\cdot)$   gives operators $\mathbf{r}(\cdot,\cdot,\cdot)$ and ${\mathbf r}_v(\cdot,\cdot,\cdot)$, in direct  analogy to (\ref{smallmlamchimaps}) and  (\ref{smallmlamchimapsv}).

  Let $f=\otimes_{v\le \infty} f_v\in I(\lambda,\chi)$ be a flat section unramified outside a finite set of places $S$ (thus $f_v|_{K_v}\equiv 1$ for $v\not\in S$).
    Then by (\ref{MandRontensors}) we have
     \begin{equation}\label{factorization of int op for r}
 \mathbf{r}(w, \lambda,\chi)f \ \ = \ \
   \left[
 \prod_{v\in S}\mathbf{r}_v(w, \lambda,\chi_{v})f_{v}\right]\cdot
\left[ \prod_{v \notin S} f_v\right].
\end{equation}
If $\beta$ is a positive simple root, then
\begin{equation}\label{factorsthroughcorootall}
\langle \lambda,\beta^\vee \rangle \ = \
\langle \lambda',\beta^\vee \rangle \ \ \Longrightarrow \ \ \begin{matrix}
  \ \ \ \ \ \ \mathbf{m}(w_\beta,\lambda,\chi) \ = \ \mathbf{m}(w_\beta,\lambda',\chi)\,,
 \\  \ \ \ \  \ \ c(w_\beta,\lambda,\chi) \ = \ c(w_\beta,\lambda',\chi)\,,
\\
\text{and}\, \ \mathbf{r}(w_\beta,\lambda,\chi) \ = \ \mathbf{r}(w_\beta,\lambda',\chi)\,;
\end{matrix}
\end{equation}
the statement for $\mathbf{m}$ follows from (\ref{Mwsimple}), the statement for $c$ is part of its definition (\ref{cvwdef}), and the statement for $\mathbf{r}$ is then a consequence of (\ref{MRcidentity}).
Because it factors through the value of $\langle \lambda,\beta^\vee\rangle$, we will sometimes write $\mathbf{r}(w_\beta,\lambda,\chi)$ as $\mathbf{r}(w_\beta,\langle \lambda,\beta^\vee \rangle,\chi)$.
As an illustration, taking the tensor product over all $v\le \infty$ of equation (\ref{Rvfactorintosimple}) shows that
\begin{multline}\label{rvfactorintosimple}
   \mathbf{r}(w,\lambda,\chi) \ \ = \ \ \mathbf{r}(w_{\beta_1},w_{\beta_2}\cdots w_{\beta_\ell}\lambda,w_{\beta_2}\cdots w_{\beta_\ell}\chi)\circ   \mathbf{r}(w_{\beta_2},w_{\beta_3}\cdots w_{\beta_\ell}\lambda,w_{\beta_3}\cdots w_{\beta_\ell}\chi)\circ
    \cdots
    \circ \mathbf{r}(w_{\beta_{\ell}},\lambda,\chi) \\ = \ \
\mathbf{r}(w_{\beta_1},\langle \lambda, w_{\beta_\ell}\cdots w_{\beta_2}\beta_1^\vee \rangle ,w_{\beta_2}\cdots w_{\beta_\ell}\chi)\circ  \mathbf{r}(w_{\beta_2},\langle \lambda, w_{\beta_\ell}\cdots w_{\beta_3}\beta_2^\vee \rangle ,w_{\beta_3}\cdots w_{\beta_\ell}\chi)\\
   \circ  \cdots
    \circ \mathbf{r}(w_{\beta_{\ell}},\langle \lambda,\beta_\ell^\vee \rangle ,\chi) \\
\end{multline}
is the composition
of ${\mathbf r}$-operators for simple reflections $w_{\beta_j}$.

\section{Langlands' square integrability condition}

The individual summands in Langlands' constant term formula \eqref{eq: constantterm} have growth rates in the positive Weyl chamber of $T(\R)$ determined by $\Re(w\lambda)\in {\frak a}_{\R}^*:=\Lambda_{\on{wt}}\otimes\R$.  Via meromorphic continuation and Laurent expansions, constant terms can be defined for leading terms  near poles or zeros of $E(f,\lambda,g)$.  Namely, suppose that $f$ is a flat section and that
$E(f,\lambda_0+\epsilon \mu,\cdot)$ has a zero of order $n$ at $\epsilon =0$ (with the convention that $n$ is negative if there is a pole), where
  $\lambda_0,\mu\in{\frak a}_{\C}^*$ are as of yet unspecified.  Then
\begin{multline}\label{limconstterm}
\qquad\qquad \lim_{\epsilon\rightarrow0}\epsilon^{-n}E(f,\lambda_0+\epsilon \mu,g) \text{~ is a nonzero automorphic form with constant term } \\  \lim_{\epsilon\rightarrow0} \sum_{w\in W}\epsilon^{-n}[M(w,\lambda_0+\epsilon \mu,\chi)f](g,\lambda_0+\epsilon \mu).\qquad\qquad\qquad
\end{multline}
It is important to note that the individual terms in this $w$-sum might not have limits as $\epsilon\rightarrow 0$.

The values of $\lambda_0$ of interest in this paper are the dominant integral weights  listed in Table~\ref{fig:biglistofcases}.  The character $\chi=\chi_0$ is also described there as the composition (\ref{globalcompwithxi}) of some fixed  Dirichlet character $\xi$ from (\ref{charset}) with the algebraic character defined by $\delta_0\in \Lambda_{\on{wt}}$.
 Let
\begin{equation}\label{WsubLdef}
W_L \ \ := \ \ \{ w  \in W\,|\,w\lambda_0=\lambda_0 \ \text{~and~} \ w\delta_0\equiv\delta_0\!\!\!\!\pmod{2\Lambda_{\on{wt}}}\}
 \end{equation}
 and
\begin{equation}\label{SigmaLdef}
\Sigma_L \ \ = \ \ \{\alpha_i\in\Sigma\,|\,\langle \lambda_0,\alpha_i^\vee\rangle =0 \ \text{~and~} \
\langle \delta_0,\alpha_i^\vee\rangle\,\in\,2\Z\}\,;
\end{equation}
  direct computation shows that these $i$  correspond to the common zero positions in the last two columns of Table~\ref{fig:biglistofcases}.  It is straightforward to verify that --   with the sole exception of the $(E_7^{sc},{\frak{sl}}_8)$ entry --  $W_L$ is the Weyl group of the standard Levi subgroup with simple roots $\Sigma_L$.   Since the $(E_7^{sc},{\frak{sl}}_8)$ case is different, it will be handled using the {\tt atlas} software and thus the following analysis does not apply to it (see the beginning of Section~\ref{sec:proofofThm4.3}).

In the remainder of this section is devoted to giving a    criteria for the square-integrability of  $ \lim_{\epsilon\rightarrow0}\epsilon^{-n}E(f,\lambda_0+\epsilon \mu,g)$ over $g\in G(\Q)\backslash G(\A)$.  
Let
\begin{equation}\label{WsuperLkostant}
\text{$W^L$ denote Kostant coset representatives for $W/W_L$,}
\end{equation}
i.e., representatives which preserve the positivity of each root in $\Sigma_L$.
Similarly to   \cite[(2.4)]{M}, group   the constant term of $E(f,\lambda_0+\epsilon \mu,g)$ according to the $W$-orbit of $\lambda_0$ as
\begin{equation}
\label{eq:constanttermgroupedbymu}
\sum_{\mu_0 \in W\! \lam_0} \sum_{\srel{w \in W}{ w\lam_0 = \mu_0}}
[M(w,\lambda_0+\epsilon \mu,\chi)f](g,\lambda_0+\epsilon \mu) \ \ = \ \
   \sum_{\mu_0 \in W \! \lam_0} \sum_{k\,\ge\,n} \epsilon^k C(\mu_0,k,f,g)\,,
   \end{equation}
 where the dependence of $C(\mu_0,k,f,g)$ on $\mu$   has been omitted from the notation for brevity.
  The sum over $k$ in (\ref{eq:constanttermgroupedbymu}) represents the Laurent series expansion in $\epsilon$ of
  the inner sum  on the left hand side, and again
\begin{equation}\label{ndef}
 \text{ $n$ is the least integer such that $C(\mu_0,n,f,\cdot)\not\equiv 0$ for some $\mu_0 \in W\lambda_0$}
\end{equation}
(in practice, many of the individual terms $C(\mu_0,k,f,g)$  vanish).
The leading coefficient $\lim_{\epsilon\rightarrow 0}\epsilon^{-n} E(f,\lambda_0+\epsilon \mu,g)$ in the Laurent expansion of $E(f,\lambda_0+\epsilon \mu,g)$ in $\epsilon$ is an automorphic form with constant term
\begin{equation}\label{constanttermofresidue}
   \sum_{\mu_0 \in W\lambda_0}  C(\mu_0,n,f,g)\,.
\end{equation}
Langlands  showed that $\lim_{\epsilon\rightarrow 0}\epsilon^{-n} E(f,\lambda_0+\epsilon \mu,g)$ is square integrable over $g\in G(\Q)\backslash G(\A)$ if and only if
\begin{equation}\label{langlandscriteria}
  \langle \mu_0,\varpi_i\rangle \ \ < \ \ 0\,, \ \ \ \ i=1,\ldots,r\,,
\end{equation}
for any $\mu_0 \in W\lambda_0$  such that $C(\mu_0,n,f,g)$  in (\ref{constanttermofresidue}) is nonzero \cite[Lemma~I.4.11]{MW}.
In particular, suppose there exists an integer $m$ (necessarily at least $n$) such that
\begin{equation}\label{ourcriteria1}
  \aligned
  \text{i)} & \ \ C(\mu_0,m,f,\cdot) \, \not\equiv\,0 \ \ \text{for some~$\mu_0 \in W\lam_0$ satisfying (\ref{langlandscriteria}), and} \\
  \text{ii)} & \ \ C(\mu_0,m',f,\cdot) \,\equiv\,0 \ \ \text{for all $m'\le m$ and all $\mu_0 \in W\lam_0$ not satisfying (\ref{langlandscriteria})}.
  \endaligned
\end{equation}
Then $\lim_{\epsilon\rightarrow 0}\epsilon^{-n} E(f,\lambda_0+\epsilon \mu,\cdot)\in L^2(G(\Q)\backslash G(\A))$.

Let $W_{\lam_0}$ denote the stabilizer of $\lam_0$ in $W$, which we recall contains $W_L$.
Fixing some $w_{\mu_0} \in W$   with $w_{\mu_0} \lam_0 = \mu_0,$
the inner sum on the left-hand side of (\ref{eq:constanttermgroupedbymu}) is
\begin{multline}\label{sumoverstabilizer}
\sum_{\srel{w \in W}{\ w\lam_0 = \mu_0}}
[M(w,\lambda_0+\epsilon \mu,\chi)f](g,\lambda_0+\epsilon \mu)
  \ \ = \ \  \sum_{w \in W_{\lambda_0}}
[M(w_{\mu_0} w,\lambda_0+\epsilon \mu,\chi)f](g,\lambda_0+\epsilon \mu) \\
= \ \
\sum_{\srel{w' \in W_{\lam_0}}{w_{\mu_0} w' \in W^L} }
\sum_{w \in W_L}
[M(w_{\mu_0} w'w,\lambda_0+\epsilon \mu,\chi)f](g,\lambda_0+\epsilon \mu)\,.
\end{multline}
Write the Laurent expansion of the inner sum as
\begin{equation}\label{Csharpdef}
\sum_{w \in W_L}
[M(w_{\mu_0} w'w,\lambda_0+\epsilon \mu,\chi)f](g,\lambda_0+\epsilon \mu)
 \ \ = \ \  \sum_{k \in\Z} \epsilon^k \,C^\sharp(w_{\mu_0} w', k,f,g)\,,
\end{equation}
where the dependence of $C^\sharp(w_{\mu_0} w', k,f,g)$ on $\lambda_0$ and $\mu$ has been omitted from the notation for brevity.
\begin{lem}
In each of the cases of Table~\ref{fig:biglistofcases} aside from the $(E_7^{sc},{\frak{sl}}_8)$ entry,
 $  C(\mu_0, k, f,\cdot)$ vanishes identically   if and only if
 $ C^\sharp(w_{\mu_0} w', k,f,\cdot)$ vanishes identically for each
 $w' \in W_{\lam_0}\cap w_{\mu_0}^{-1}W^L$.
\end{lem}
\begin{proof}
Let $T(\A)^1=\{ t \in T(\A)\mid |t^\lambda|= 1, \forall  \lambda \in \Lambda_{\on{wt}}\}$, where $|\cdot|$ denotes the (global) adelic valuation.  For $t\in T(\A)^1$ we in particular have that $t^\alpha$ has   adelic valuation 1 for any root $\alpha$, and that $|t|^\lambda$ appearing in (\ref{Idef}) is trivial for all $\lambda \in {\frak a}_\C^*$.
For each $w' \in W_{\lam_0}\cap w_{\mu_0}^{-1}W^L$ we have
$$
C^\sharp(w_{\mu_0} w', k,f,tg)
 \ \ = \ \  (w_{\mu_0} w'\chi)(t)\,C^\sharp(w_{\mu_0} w', k,f,g)\, , \   t \in T(\A)^1 \ \text{and} \ g \in G(\A)\,.
$$
From its definition in (\ref{eq:constanttermgroupedbymu}), $C(\mu_0,k,f,g)$ is the coefficient of $\epsilon^k$ in the Laurent expansion of the left-hand side of (\ref{sumoverstabilizer}), so
\begin{equation}\label{Ccsharpsum}
  C(\mu_0,k,f,g) \ \ = \ \ \sum_{\srel{w' \in W_{\lam_0}}{w_{\mu_0} w' \in W^L} } C^\sharp(w_{\mu_0} w', k,f,g)
\end{equation}
and
\begin{equation}\label{Ccsharpsumt}
  C(\mu_0,k,f,tg) \   = \   \sum_{\srel{w' \in W_{\lam_0}}{w_{\mu_0} w' \in W^L} } (w_{\mu_0} w'\chi)(t)\,C^\sharp(w_{\mu_0} w', k,f,g)\,, \ \ \ t\in T(\A)^1\,.
\end{equation}
The Lemma follows once we show that the restrictions of the characters $\{ w_{\mu_0} w'\chi : w'\in W_{\lambda_0}\cap w_{\mu_0}^{-1}W^L\}$ to $T(\A)^1$ are  distinct (and hence linearly independent).
First suppose  that there is an element  $w''\in W_{\lambda_0}\cap w_{\mu_0}^{-1}W^L$
distinct from $w'$ (and hence  by (\ref{WsuperLkostant}) in different right $W_L$-cosets) for which $\chi'=w_{\mu_0} w'\chi$ and $\chi''=w_{\mu_0} w''\chi$ are equal as characters on $T(\A)$.  Then $(w')^{-1}w''$ stabilizes both $\lambda_0$ and $\chi$, whose common stabilizer was observed after (\ref{SigmaLdef}) to be $W_L$ -- a contradiction.  If the restrictions of distinct characters $\chi'$ and $\chi''$ to $T(\A)^1$ agree, then $(\chi')^{-1}\chi''$ is a nontrivial quadratic character on $T(\A)$ which is trivial on $T(\A)^1$.  Every idele is the product of a positive real number and an adele having valuation 1.  Using  the unique factorization (\ref{uniquefact}), it follows that $T(\A)$ is the product of $T(\A)^1$ and a connected subgroup of $T(\R)$.  This connectedness forces $(\chi')^{-1}\chi''$ to be constant on all of $T(\A)$, a contradiction.

\end{proof}

The Lemma motivates regrouping the constant term of $E(f,\lambda_0+\epsilon \mu,g)$   according
to $W_L$ cosets:
\begin{equation}\label{groupedconstantterm}
  \sum_{w'\in W^L} \sum_{w\,\in\,W_L}[M(w'w,\lambda_0+\epsilon \mu,\chi)f](g,\lambda_0+\epsilon \mu) \ \ = \ \
   \sum_{w'\in W^L} \sum_{k\,\ge\,n} \epsilon^k C^\sharp(w',k,f,g)\,,
\end{equation}
and the square-integrability condition for   $\lim_{\epsilon\rightarrow 0}\epsilon^{-n} E(f,\lambda_0+\epsilon \mu,g)$ given in (\ref{ourcriteria1}) is itself implied by the existence of an integer $m$ such that
\begin{equation}\label{ourcriteria2}
  \aligned
  \text{i)} & \ \ C^\sharp(w',m,f,\cdot) \, \not\equiv\,0 \ \ \text{for some~$w'\in W^L$ with $\mu_0=w' \lam_0$ satisfying (\ref{langlandscriteria}), and} \\
  \text{ii)} & \ \ C^\sharp(w',m',f,\cdot) \,\equiv\,0 \ \ \text{for all $m'\le m$ and all $w'\in W^L$ with $\mu_0=w'\lambda_0$ not  satisfying (\ref{langlandscriteria})}.
  \endaligned
\end{equation}
In the next section we will study   particular deformation directions $\mu$ and use the factorization (\ref{MRcidentity}), which allows us to leverage arithmetic information from the $c$-functions (\ref{cwdef}).

\section{Proof of Theorem~\ref{thm:unitarizablequotient}}\label{sec:proofofThm4.3}

According to the reduction to distinguished parameters in Section~\ref{sec:sub:reductiontodistinguished}, it suffices to establish the square integrability of (\ref{limconstterm}) for the cases listed in Table~\ref{fig:biglistofcases}.  In order to streamline the presentation\footnote{Our argument treats all but these four cases uniformly.  The {\tt atlas} software also can also demonstrate the unitarity of all the Arthur unipotent representations on real forms of $G_2$, $F_4$, $E_6$, and (on very large memory servers) $E_7$, but some  of the remaining $E_8$ examples appear to be well beyond its present capabilities.} we also used the  {\tt{is\_unitary}} command in {\tt{atlas}} to show the unitarity of four more cases:~the last $E_7$ case (in which $\mathcal O^\vee=E_6(a_1)$); the first $E_8$ case (in which $\mathcal O^\vee=D_5+A_2$); and the two $E_8$ cases in which $\mathcal O^\vee=E_8(a_7)$.

It follows easily from direct computation   that in each of our remaining cases from  Table~\ref{fig:biglistofcases}, $W_L$ from (\ref{WsubLdef}) is isomorphic to $(\Z/2\Z)^{\Sigma_L}$.
 We shall show the nonvanishing and square-integrability of $\lim_{\epsilon\rightarrow 0}\epsilon^{-n} E(f,\lambda_0+\epsilon \mu,g)$ from (\ref{limconstterm}).  In particular, we will show in each case that there is some value of $m$ such that  part i) of (\ref{ourcriteria2}) holds for some particular global flat section $f=\otimes_{v\le \infty}f_v$,  while also   part ii) of (\ref{ourcriteria2}) holds for all global flat sections  -- not just $f$ -- and the same value of $m$.  Infinitesimal right translation by $\frak g(\R)$
  on the image of the map $f'_\infty\mapsto \lim_{\epsilon\rightarrow0}\epsilon^{-n}E(f'_\infty\otimes\bigotimes_{v<\infty}f_v,\lambda_0+\epsilon\mu,\cdot)$ yields a
  quotient of the $(\frak{g},K)$-module $I_\infty(\lambda_0,\chi_\infty)$; the square-integrability of the image then shows that the $L^2$-inner product on the Eisenstein series gives a unitary structure on this quotient.

  Consider the inner sum in (\ref{groupedconstantterm}) evaluated at $g=bk$, where $b\in B(\A)$ and $k\in K$:
\begin{multline}\label{dominnersum1}
 \sum_{k\,\ge\,n} \epsilon^k \, C^\sharp(w',k,f,bk) \ \ = \ \
\sum_{w\,\in\,W_L}(w'w\chi)(b)|b|^{w'w(\lambda_0+\epsilon\mu)+\rho}[M(w'w,\lambda_0+\epsilon \mu,\chi)f](k,\lambda_0+\epsilon \mu)
\\
 = \ \ |b|^{w'\lambda_0 +\rho}(w'\chi)(b)
\sum_{w\,\in\,W_L}|b|^{\epsilon w'w  \mu}[M(w'w,\lambda_0+\epsilon \mu,\chi)f](k,\lambda_0+\epsilon \mu) \,.
\end{multline}
From this point onwards the flat section $f$ will be evaluated only at the fixed point $k$ (which could be taken to be the identity after right-translating $f$ anyhow), so we shall  drop the dependence on $k$ and $\lambda$ from the argument of $f$ and use the maps ${\mathbf{m}}(w,\lambda,\chi)$ from (\ref{smallmlamchimaps}).
  Thus the last sum over $w\in W_L$  is
\begin{equation}\label{dominnersum2}
\sum_{w\,\in\,W_L}|b|^{\epsilon w'w  \mu}\,\mathbf{m}(w'w,\lambda_0+\epsilon \mu,\chi)f \,,
\end{equation}
for example.

Recall the factorization (\ref{MRcidentity}) and that $R(w'w,\lambda,\chi)$ is holomorphic for $\lambda$ near the dominant weight $\lambda_0$ by part \ref{int:8} of Theorem~\ref{newintertwintheorem}.  Thus all the poles of $\mathbf{m}(w'w,\lambda_0+\epsilon \mu,\chi)$ at $\epsilon=0$ come from $c(w'w,\lambda_0+\epsilon \mu,\chi)$, which was defined as a product of functions  $c(s)$ and $c(s,\xi)$  in (\ref{cwdef}), where $\xi$ is a nontrivial Dirichlet character chosen from the set $\mathcal C$ (\ref{charset}).  Since $\lambda_0$ is an integral weight,   the values of $s$ here are all near  integers.
Motivated by the analytic characterization in Lemma~\ref{lem:Lfunctionproperties} of these functions at integers,  we introduce the following notation for $|s-n|<\frac 12$:
\begin{equation}\label{domctildes}
   \tilde{c}_n(s,\xi ) \ \ = \ \ c(n+s,\xi ) \ \  \ \ \ \text{and} \ \ \  \ \   \tilde{c}_n(s) \ \ = \ \ \left\{
                               \begin{array}{ll}
                                 c(n+s), & n\neq \pm 1, \\
                                 s\,c(1+s), & n=1,\\
                                  s^{-1}\,c(-1+s), & n=-1.
                               \end{array}
                             \right.
\end{equation}
where $n\in \Z$;
these functions are each  holomorphic and nonvanishing at $s=0$.
In analogy to (\ref{cwdef}), let
\begin{equation}\label{domcwtildes}
    \tilde{c}(w,\lambda_0+\epsilon \mu,\chi) \ \ = \ \ \prod_{ \alpha>0,w\alpha<0} \tilde{c}_{\langle \lambda_0,\alpha^\vee \rangle}(\epsilon \langle \mu,\alpha^\vee\rangle,\chi\circ \alpha^\vee)
\end{equation}
so that
\begin{equation}\label{relationbetweencs}
    c(w,\lambda_0+\epsilon \mu,\chi) \ \ = \ \ \epsilon^{o_{w}}
\left[
\prod_{\alpha\,\in\,{\mathcal S}(w)} \langle \mu,\alpha^\vee \rangle^{-1}\right]
 \tilde{c}(w,\lambda_0+\epsilon \mu,\chi)\,,
\end{equation}
where
\begin{equation}\label{domcalSdef}
    {\mathcal S}(w) \ \ = \ \ \{\alpha >0 \ | \  w\alpha<0\,, \ \langle \delta_0,\alpha^\vee\rangle \in 2\Z\,, \ \text{~and~}\ \langle \lambda_0,\alpha^\vee\rangle =1\}\,,
\end{equation}
and $-o_{w} = \,\#{\mathcal S}(w)$ is the order of the pole of $c(w,\lambda_0+\epsilon\mu,\chi)$ at $\epsilon=0$ (it comes from factors in (\ref{domcwtildes}) with $\chi\circ\alpha^\vee$ trivial and $\langle \lambda_0,\alpha^\vee\rangle=1)$.  For $\mu$ in general position (such as $\mu=\rho$), the inner products $\langle \mu,\alpha^\vee\rangle$ in (\ref{relationbetweencs}) are nonvanishing.

Set
  \begin{equation}\label{dommtildedef}
    \tilde{\mathbf{m}}(w,\lambda,\chi)  \ \ := \ \  \tilde{c}(w,\lambda,\chi)\,\mathbf{r}(w,\lambda,\chi)
  \end{equation}
  and write  $\tilde{\mathbf{m}}(w_\beta,\lambda,\chi)$ as $\tilde{\mathbf{m}}(w_\beta,\langle \lambda,\beta^\vee\rangle,\chi)=\tilde{c}(w_\beta,\lambda,\chi)\,\mathbf{r}(w_\beta,\langle \lambda,\beta^\vee\rangle,\chi)$ when $\beta\in \Sigma$ (extending the convention for $\mathbf{r}(w_\beta,\lambda,\chi)$ introduced after (\ref{factorsthroughcorootall})).
We then deduce from (\ref{relationbetweencs}), followed by an application of (\ref{domcwtildes}) and (\ref{rvfactorintosimple}), that
\begin{multline}\label{dommtildesonf}
    \mathbf{m}(w',\lambda_0+\epsilon\mu,\chi) \ \ = \ \ \epsilon^{o_{w'}}
\left[
\prod_{\alpha\,\in\,{\mathcal S}(w')} \langle \mu,\alpha^\vee \rangle^{-1}
\right] \tilde{\mathbf{m}}(w',\lambda_0+\epsilon\mu,\chi) \\
\!\!\!\!\!\!\!\!\!\!\!\!\!\!\!\!\!\!\!\!\!\!\!\!\!\!\!\!\!\!\!\!\!\!\!\!\!\!\!\!\!\!\!\!\!  = \ \
\epsilon^{o_{w'}}
\left[
\prod_{\alpha\,\in\,{\mathcal S}(w')} \langle \mu,\alpha^\vee \rangle^{-1}
\right] \tilde{\mathbf{m}}(w_{\beta_1},\langle \lambda_0,\gamma_1^\vee \rangle+\epsilon \langle \mu,\gamma_1^\vee\rangle,w_{\beta_2}\cdots w_{\beta_\ell}\chi)\circ \\
\tilde{\mathbf{m}}(w_{\beta_2},\langle \lambda_0,\gamma_2^\vee \rangle+\epsilon \langle \mu,\gamma_2^\vee\rangle,w_{\beta_3}\cdots w_{\beta_\ell}\chi)\circ \cdots \circ
\tilde{\mathbf{m}}(w_{\beta_\ell},\langle \lambda_0,\gamma_\ell^\vee \rangle+\epsilon \langle \mu,\gamma_\ell^\vee\rangle,\chi)\,,
\end{multline}
where $w'\in W^L$ is written as a reduced word $w_{\beta_1}\cdots w_{\beta_\ell}$ of length $\ell=\ell(w')$, $\beta_i\in\Sigma$, and
\begin{equation}\label{domgammaidef}
    \gamma_i^\vee \ \ = \ \ w_{\beta_\ell}\cdots w_{\beta_{i+1}}\beta_i^\vee \ \, , \ \ \ \ 1\,\le\,i\,\le\,\ell\,,
\end{equation}
are the positive coroots whose sign is flipped by $w'$.

We now introduce the   notation
\begin{equation}\label{domderivnotation}
    \tilde{\mathbf{m}}^{(n)}_{\nu}(w,\lambda,\chi) \ \ = \ \ \left.\frac{1}{n!}\,\frac{\partial^n}{\partial \epsilon^n}\right|_{\epsilon=0}  \tilde{\mathbf{m}}(w,\lambda+\epsilon\nu,\chi)\,, \ \ \nu\,\in\,{\mathfrak{a}}_\C^*\,,
\end{equation}
for the Taylor coefficients of intertwining operators (keeping in mind that $ \tilde{\mathbf{m}}(w,\lambda,\chi)$ is a meromorphic  matrix-valued function on any $K$-finite subspace of $\Ind_{B(\A) \cap K}^K \chi$).
In the special case that $w=w_{\beta}$ is the simple Weyl reflection associated to $\beta\in \Sigma$  and $\varpi_\beta$ denotes the corresponding fundamental weight,
\begin{equation}\label{domsimplederivs}
     \tilde{\mathbf{m}}^{(n)}_{\nu}(w_\beta,\lambda,\chi) \ \ = \ \ \langle \nu,\beta^\vee\rangle^n\,  \tilde{\mathbf{m}}^{(n)}_{\varpi_\beta}(w_\beta,\lambda,\chi)\,,
\end{equation}
since $\tilde{\mathbf{m}}(w_\beta,\lambda,\chi)$ depends only on $\langle\lambda,\beta^\vee\rangle$
(cf.~the comments just after (\ref{dommtildedef})).  It is clear from (\ref{relationbetweencs}) and (\ref{dommtildedef}) that  $\mathbf{m}(w,\lambda_0+\epsilon \mu,\chi)$ and  $\tilde{\mathbf{m}}(w,\lambda_0+\epsilon \mu,\chi)$ coincide when $\mathcal{S}(w)$ is empty, as is the case when $w\in W_L$ (this is because such $w$ can only flip the signs of roots in the Levi component generated by $\Sigma_L$, all of which are orthogonal to $\lambda_0$ by (\ref{SigmaLdef})).
Thus for $w\in W_L$ we may write
\begin{equation}\label{domorder+++2}
 \mathbf{m}(w,\lambda_0+\epsilon \mu,\chi)f \  = \ \sum_{n\,\ge\,0}\epsilon^n \, \tilde{\mathbf{m}}^{(n)}_{\mu}(w,\lambda_0,\chi)f\,, \ \ \tilde{\mathbf{m}}^{(0)}_{\mu}(w,\lambda_0,\chi) \  =  \ (-1)^{\ell(w)},
\end{equation}
the latter equation following from (\ref{WsubLdef})-(\ref{SigmaLdef}) and part~\ref{int:13} of Theorem~\ref{newintertwintheorem} (cf.~also the beginning of its proof).

We now insert  (\ref{dommtildesonf}) and  (\ref{domorder+++2}) into (\ref{dominnersum2}) to obtain
\begin{multline}\label{domorder+++4}
\sum_{w\,\in\,W_L}|b|^{\epsilon w'w  \mu}\,\tilde{\mathbf{m}}(w'w,\lambda_0+\epsilon \mu,\chi)f \ \ = \ \
\sum_{w\,\in\,W_L}|b|^{\epsilon w'w  \mu}\,\tilde{\mathbf{m}}(w',\lambda_0+\epsilon w\mu,\chi)\,\tilde{\mathbf{m}}(w,\lambda_0+\epsilon \mu,\chi)f \\
 = \ \
\sum_{w\,\in\,W_L}|b|^{\epsilon w'w  \mu}\,
\epsilon^{o_{w'}}\left[
\prod_{\alpha\,\in\,{\mathcal S}(w')} \langle w\mu,\alpha^\vee \rangle^{-1}
\right]       \tilde{\mathbf{m}}(w_{\beta_1},\langle \lambda_0,\gamma_1^\vee \rangle+\epsilon \langle w\mu,\gamma_1^\vee\rangle ,w_{\beta_2}\cdots w_{\beta_\ell}\chi)\\\circ \cdots \circ
\tilde{\mathbf{m}}(w_{\beta_\ell},\langle \lambda_0,\gamma_\ell^\vee \rangle+\epsilon \langle w\mu,\gamma_\ell^\vee\rangle ,\chi)
\sum_{n\,\ge\,0}\epsilon^n\,\tilde{\mathbf{m}}^{(n)}_{\mu}(w,\lambda_0,\chi)f
\\
=
\sum_{w\,\in\,W_L} \left[
\prod_{\alpha\,\in\,{\mathcal S}(w')} \langle w\mu,\alpha^\vee \rangle^{-1}
\right]
\sum_{n,n_0,\ldots,n_\ell\ge 0} \epsilon^{o_{w'}+n+n_0+\cdots+n_\ell} \frac{\langle w'w \mu,\log|b|\rangle^{n_0}}{n_0!}
\langle w\mu,\gamma_1^\vee\rangle^{n_1}\cdots\langle w\mu,\gamma_\ell^\vee\rangle^{n_\ell} \\
 \tilde{\mathbf{m}}^{(n_1)}_{\varpi_{\beta_1}}
(w_{\beta_1},\langle \lambda_0,\gamma_1^\vee \rangle  ,w_{\beta_2}\cdots w_{\beta_\ell}\chi)\circ \cdots \circ
\tilde{\mathbf{m}}^{(n_\ell)}_{\varpi_{\beta_\ell}} (w_{\beta_\ell},\langle \lambda_0,\gamma_\ell^\vee \rangle ,\chi)\circ \tilde{\mathbf{m}}^{(n)}_{\mu}(w,\lambda_0,\chi)f
\end{multline}
for $w'\in W^L$, where in the last step we have invoked (\ref{logbdef})  and  inserted the power series expansion of each intertwining operator using (\ref{domderivnotation})-(\ref{domsimplederivs}).

At this point we  use the observation (as evidenced by Table~\ref{fig:biglistofcases}) that all our remaining cases have $W_L\cong (\Z/2\Z)^{\#\Sigma_L}$, $\#\Sigma_L \le 3$, defined in (\ref{WsubLdef})
is generated by the commuting simple reflections $\{w_{\alpha}|\alpha\in \Sigma_L\}$.  Thus elements of $W_L$ are in bijective correspondence with subsets $S_L$ of $\Sigma_L$, each of which defines an element $w_{S_L}$ defined by the commuting product $\prod_{\beta\in S_L}w_{\beta}$.
An application of (\ref{WsubLdef}) and formula (\ref{factorsthroughcorootall}) then shows that $\tilde{\mathbf{m}}(w_\beta,w_{S_L}(\lambda_0+\epsilon \mu),w_{S_L}\chi)=\tilde{\mathbf{m}}(w_\beta, \lambda_0+\epsilon \mu, \chi)$  if $\beta\in \Sigma_L-S_L$, since then $ \langle w_{S_L}\mu,\beta^\vee\rangle=\langle \mu,w_{S_L}\beta^\vee\rangle=\langle \mu,\beta^\vee\rangle$ and $w_L$ fixes both $\lambda_0$ and $\chi$.  Using the analog of (\ref{langlandslemma}) for $\tilde{\mathbf{m}}$ as in (\ref{dommtildesonf}), we obtain
\begin{equation}\label{domorder+++5}
    \tilde{\mathbf{m}}\left(\textstyle{\prod}_{\beta\,\in\, S_L}w_\beta,\lambda_0+\epsilon\mu,\chi\right) \ \ = \ \
 \prod_{\beta\,\in\, S_L}  \tilde{\mathbf{m}}(w_\beta,\lambda_0+\epsilon\mu,\chi) \,,
\end{equation}
where the products on both sides are commutative.
 Summing over nonnegative integers $n_\beta$ parameterized by the elements $\beta\in \Sigma_L$, it follows that
\begin{equation}\label{domorder+++6}
   \tilde{ \mathbf{m}}^{(n)}_{\mu}(w_{S_L},\lambda_0,\chi) \ \  =  \ \ \sum_{\srel{\srel{n_\beta\,\ge \,0 }{\beta\notin S_L\Longrightarrow n_\beta =0}}{\sum_{\beta\in \Sigma_L}n_\beta = n}} \prod_{\beta\,\in\,S_L} \tilde{\mathbf{m}}^{(n_\beta)}_{\mu}(w_\beta,\lambda_0,\chi)
\end{equation}
 for any subset $S_L\subset \Sigma_L$.
Inserting into (\ref{domorder+++4}), we see that (\ref{dominnersum2}) equals
\begin{multline}\label{domorder+++7}
\sum_{\srel{\srel{n_0,\ldots,n_\ell\ge 0}{n_\beta\,\ge\,0}}{ \{\beta\in\Sigma_L|n_\beta>0\} \, \subseteq \, S_L \, \subseteq \, \Sigma_L}}
 \epsilon^{o_{w'}+n_0+\cdots
+n_\ell+\sum_{\beta\in\Sigma_L}n_\beta}
   \left[
\prod_{\alpha\,\in\,{\mathcal S}(w')} \langle w_{S_L}\mu,\alpha^\vee \rangle^{-1}
\right] \\ \times \frac{\langle w'w_{S_L} \mu,\log|b|\rangle^{n_0}}{n_0!}
\left[\prod_{j=1}^\ell \langle w_{S_L}\mu,\gamma_j^\vee\rangle^{n_j}
\right]
 \tilde{\mathbf{m}}^{(n_1)}_{\varpi_{\beta_1}} (w_{\beta_1},\langle \lambda_0,\gamma_1^\vee \rangle   ,w_{\beta_2}\cdots w_{\beta_\ell}\chi)\\\circ \cdots \circ
\tilde{\mathbf{m}}^{(n_\ell)}_{\varpi_{\beta_\ell}} (w_{\beta_\ell},\langle \lambda_0,\gamma_\ell^\vee \rangle  ,\chi)\circ\prod_{\beta\,\in\,S_L} \tilde{\mathbf{m}}^{(n_\beta)}_{\mu}(w_\beta,\lambda_0,\chi)f.
\end{multline}
This is the Laurent expansion of (\ref{dominnersum2}) at $\epsilon=0$.

Consider the contribution of terms in (\ref{domorder+++7}) for fixed values of  $n_j\ge 0$  and $n_\beta\ge 0$, where $0\le j\le \ell$ and $\beta\in\Sigma_L$, but varying choices of $S_L$.  The minimal choice of $S_L$ is $\{\beta\in \Sigma_L|n_\beta>0\}$, but larger choices of $S_L$ will involve the operators   $\tilde{\mathbf{m}}^{(0)}(w_\beta,\lambda_0,\chi)$, which act by $-1$ according to the second formula in (\ref{domorder+++2}).  However, all the other  intertwining operators that appear in (\ref{domorder+++7}) are  independent of $S_L$.

 This last observation   gives a combinatorial way to verify (\ref{ourcriteria2}).
 If
\begin{multline}\label{domvanishcondition}
\sum_{ \{\beta\in\Sigma_L|n_\beta>0\} \, \subseteq \, S_L \, \subseteq \, \Sigma_L}(-1)^{\#S_L}
 \left[
\prod_{\alpha\,\in\,{\mathcal S}(w')} \langle w_{S_L}\mu,\alpha^\vee \rangle^{-1}
\right] \langle w'w_{S_L} \mu,\log|b|\rangle^{n_0} \\ \times \
\langle w_{S_L}\mu,\gamma_1^\vee\rangle^{n_1}\cdots\langle w_{S_L}\mu,\gamma_\ell^\vee\rangle^{n_\ell}
 \ \ = \ \ 0
\end{multline}
for all choices of $n_j\ge 0$ and $n_\beta\ge 0$ (where $0\le j \le \ell$ and $\beta\in \Sigma_L$)  satisfying the constraint
\begin{equation}\label{constraint}
  n_0\, + \, \cdots \, + \, n_\ell \, + \, \sum_{\beta\,\in\,\Sigma_L}n_\beta \ \ \le \  k_{\text{bd}}\,,
\end{equation}
 then (\ref{dominnersum2})=(\ref{domorder+++7}) vanishes to order at least $o_{w'}+k_{\text{bd}}+1$.   In terms of (\ref{dominnersum1}), this shows that $C^\sharp(w',o_{w'}+k_{\text{bd}},f,\cdot)\equiv 0$ for all $f$.
On the other hand, suppose that (\ref{domvanishcondition}) does not vanish in the special case in which $n_0=k_{\text{bd}}+1$, but all other $n_j$ and $n_\beta=0$ -- that is,
\begin{equation}\label{domnonvanishcondition}
    \sum_{   S_L \, \subseteq \, \Sigma_L}(-1)^{\#S_L}
 \left[
\prod_{\alpha\,\in\,{\mathcal S}(w')} \langle w_{S_L}\mu,\alpha^\vee \rangle^{-1}
\right] \langle w'w_{S_L} \mu,\log|b|\rangle^{k_{\text{bd}}+1} \ \ \neq \ \ 0
\end{equation}
for some $b\in B(\A)$.
Again using $\tilde{\mathbf{m}}_\mu^{(0)}(w_\beta,\lambda_0,\chi)f=-f$ from (\ref{domorder+++2}), the contribution
of terms for these vanishing $n_j$ and $n_\beta$ in
(\ref{domorder+++7}) is a nonzero multiple of (\ref{domnonvanishcondition}) times $\epsilon^{o_{w'}+k_{\text{bd}}+1}\tilde{\mathbf{m}}(w',\lambda_0,\chi)f$.  According to part \ref{int:14} of Theorem~\ref{newintertwintheorem}, there is some choice of $f$ such that $\mathbf{r}(w',\lambda_0,\chi)f\neq 0$; for such $f$, $\tilde{\mathbf{m}}(w',\lambda_0,\chi)f$ is nonzero as well since the two intertwining operators are related in (\ref{dommtildedef}) by the nonvanishing factor $\tilde{c}(w',\lambda_0,\chi)$ (see (\ref{domctildes})-(\ref{domcwtildes})).  We conclude that (\ref{domnonvanishcondition}) implies that
(\ref{dominnersum2}) vanishes to order precisely $o_{w'}+k_{\text{bd}}+1$ and  that $C^\sharp(w',o_{w'}+k_{\text{bd}}+1,f,\cdot)\not\equiv 0$, because any other choice of $n_j$ and $n_\beta$ with $n_0+\cdots+n_\ell+\sum_{\beta\in\Sigma_L}n_\beta=k_{\text{bd}}+1$ has a strictly smaller value of  $n_0$ (the power $\langle w'w_{S_L}\mu,\log|b|\rangle$), and hence cannot cancel out (\ref{domnonvanishcondition}).  This proves Theorem~\ref{thm:unitarizablequotient}.

The vanishing condition (\ref{domvanishcondition}) is a statement about multilinear forms.  Using the notation $v^{\otimes k}$ for the $k$-fold tensor product of a vector with itself (so that $v^{\otimes 0}=1$, $v^{\otimes 1}=v$, and $v^{\otimes 2}=vv^t$, which are the only cases used in the final calculation), we observe that
\begin{equation}\label{domvanishallnew}
   \sum_{ \{\beta\in\Sigma_L|n_\beta>0\} \, \subseteq \, S_L \, \subseteq \, \Sigma_L}(-1)^{\#S_L}
 \left[
\prod_{\alpha\,\in\,{\mathcal S}(w')} \langle w_{S_L}\mu,\alpha^\vee \rangle^{-1}
\right] (w_{S_L} \mu)^{\otimes k} \ \ = \ \ 0\,,
\end{equation}
for all $0\le k\le k_{\text{bd}}-\sum_{\beta\in \Sigma_L}n_\beta$, implies the vanishing of (\ref{domvanishcondition}) for all $n_j$ and $n_\beta$ satisfying (\ref{constraint}).  Condition (\ref{domnonvanishcondition}) can be similarly rephrased, but it is simpler to computationally verify it for $\log|b|=\rho$.

The  proof now reduces to a fairly   simple calculation to verify (\ref{domnonvanishcondition}) and (\ref{domvanishallnew}), which is easily implemented in standard packages such as Mathematica or LiE.  We took the deformation direction to be  $\mu=\rho$.     As Table~\ref{fig:detailstable} indicates, our verification required taking $k_{\text{bd}}\le\#\Sigma_L-1$ and $m\le -r-\#\Sigma_L$, where we recall that $r$ is the rank of $G$ and that $W_L$ is isomorphic to $(\Z/2\Z)^{\#\Sigma_L}$.  In particular, when $W_L$ is trivial the computation simplifies tremendously because there is no possible cancellation between terms in (\ref{dominnersum2}).
This is why those entries are marked ``NA'', as are the four entries for which   unitarity was separately handled by {\tt atlas}.

\begin{table}[h]
$$\begin{array}{|c|c||c|c|c|}
\hline
\text{$(G,\frak{g^\vee_\C}^\sigma)$}&
\text{Orbit}&
\text{$W_L$ Type} &
m & k_{\text{bd}}\le
\\
\hline
(G_2,\ \f{sl}_2\times \f{sl}_2)&G_2(a_1) &1 &  -2   & NA  \\
\hline
(F_4,\ \f{sp}_6 \times \f{sl}_2) &F_4(a_3)&A_1 & -5& 0  \\
(F_4,\ \f{sp}_6 \times \f{sl}_2 )&F_4(a_2)&1 & -4& NA  \\
(F_4,\ \f{so}_9) &F_4(a_3)&A_1\times A_1 & -6& 1  \\
(F_4,\ \f{so}_9) &F_4(a_1)&1 & -4 & NA  \\
\hline
(E_6^{sc}, \  \f{sl}_6 \times \f{sl}_2) & E_6(a_3)& 1 & -6& NA  \\
\hline
(E_7^{sc},\ \f{so}_{12} \times \f{sl}_2) & E_7(a_5)&A_1\times A_1 & -9 & 1  \\
(E_7^{sc},\ \f{so}_{12} \times \f{sl}_2) & E_7(a_4)&A_1 & -8& 0  \\
(E_7^{sc},\ \f{so}_{12} \times \f{sl}_2) & E_7(a_3)&1 & -7& NA  \\
(E_7^{sc},\ \f{sl}_8)  & E_6(a_1) & A_1 & NA & NA \\
\hline
(E_8,\ \f{e}_7\times \f{sl}_2)& D_5+A_2& A_2\times A_1\times A_1 & NA  & NA    \\
(E_8,\ \f{e}_7\times \f{sl}_2)& D_7(a_1)   &A_1\times A_1\times A_1& -11  & 2   \\
(E_8,\ \f{e}_7\times \f{sl}_2)& E_8(a_7)&A_2\times A_2\times A_1 & NA& NA  \\
(E_8,\ \f{e}_7\times \f{sl}_2)& E_8(b_5)&A_1\times A_1 & -10& 1  \\
(E_8,\ \f{e}_7\times \f{sl}_2)& E_8(b_4)& A_1 & -9& 0  \\
(E_8,\ \f{e}_7\times \f{sl}_2)& E_8(a_3)&1 & -8& NA  \\
(E_8,\ \f{so}_{16})& E_8(a_7)&A_2\times A_1\times A_1\times A_1 & NA& NA  \\
(E_8,\ \f{so}_{16})& E_8(b_6)&A_1\times A_1\times A_1 &  -11& 2  \\
(E_8,\ \f{so}_{16})& E_8(a_6)&A_1\times A_1 & -10 & 1  \\
(E_8,\ \f{so}_{16})& E_8(a_5)&A_1 & -9& 0  \\
(E_8,\ \f{so}_{16})& E_8(a_4)&1 & -8& NA  \\
\hline
\end{array} $$
\caption{$W_L$ along with the values  of  $m$ and largest necessary $k_{\text{bd}}$ for each case listed in Table~\ref{fig:biglistofcases}.  The present argument applies to the  cases in which the $W_L$ type is a product of $A_1$ factors. \label{fig:detailstable}}
\end{table}

\noindent {\bf Remarks:}

1) It is tempting to follow the strategy of \cite{M} and use the induction data $(\lambda_1,\delta_1)$ from Table~\ref{fig:biglistofcases}, with which one might expect that the vast majority of terms in (\ref{limconstterm}) might vanish identically (with a suitably-chosen deformation direction $\mu$).  This is because among its Weyl orbit, $\lambda_1$ has the maximal number of inner products with simple coroots equal to -1, a point at which the $c$-function vanishes.  Such vanishing was a crucial tool in \cite{M}, as it immediately eliminated most $w\in W$ from the decisive  verification of (\ref{ourcriteria1}).  However, applying this argument in our context faces an obstacle not present there (where only the spherical vector was relevant):~if one writes  $f=M(w_0,\lambda_0,\chi_0)f'$, with $f'$ a flat section and $w_0\lambda_0=\lambda_1$, it appears difficult to simultaneously deform both $f$ and $f'$ in a way which allows us to leverage the  composition formula (\ref{langlandslemma}).

2) In \cite{opdam}, an approach was introduced to relate the constant-term combinatorics to affine Hecke-algebras, in particular explaining the cancelation found in \cite{M}.  It would be intriguing to understand whether (\ref{domvanishcondition}) can also be understood by techniques similar to those used in \cite[Lemma~2.17]{opdam}.


\begin{thebibliography}{AAAAA}


\bibitem[ABV]{ABV} J. Adams, D. Barbasch, and D. Vogan, {\it The Langlands Classication and the Irreducibile Characters for Real Reductive groups}, Progress in Mathematics {\bf 104}, Birkhauser, 1992.

\bibitem[ALTV]{The5} J. Adams, M.~van Leeuwen, P. Trapa, and D.A. Vogan Jr., {\it Unitary representations of real reductive groups},   arXiv:1212.2192 (2012).



\bibitem[A1]{Arthur-conjectures1}  J. Arthur, {\it On some problems suggested by the trace formula}, in
Lie Group Representations II, Lecture Notes in Math., {\bf 1041}, Springer-Verlag, 1983, 1--49.

\bibitem[A2]{Arthur-conjectures2}  J. Arthur, {\it Unipotent automorphic representations:~conjectures},
Ast\'erisque  {\bf 171-172} (1989), 13--71.

\bibitem[A3]{Arthur-IntOpI}  J. Arthur, {\it
Intertwining operators and residues. I. Weighted characters},
J. Funct. Anal. {\bf 84} (1989), no. 1, 19--84.

\bibitem[A4]{Arthur-book} J. Arthur, {\it The Endoscopic Classification of Representations Orthogonal and Symplectic Groups}.   American Mathematical Soc. Colloquium Publications {\bf 61}, 2013.

\bibitem[atlas]{atlas} Atlas of Lie Groups and Representations, \url{http://www.liegroups.org/}

\bibitem[BC]{BC} D. Barbasch and D. Ciubotaru, {\it Unitary equivalences for reductive $p$-adic groups}, American Journal of Mathematics {\bf 135} (2013), 1633--1674.

\bibitem[BM]{BM} D. Barbasch and A. Moy, {\it Reduction to real infinitesimal character in affine Hecke algebras}, J. Amer. Math. Soc. {\bf 6}
(1993), no. 3, 611--635.

\bibitem[BB]{BB} A. Beilinson and J. Bernstein,
{\it Localisation de $\frak g$-modules}, C. R. Acad. Sci. Paris {\bf 292}
(1981), 15--18.


\bibitem[Cass1]{Casselman}
W.~Casselman, {\it Introduction to the theory of admissible representations
of $p$-adic reductive groups}, Available at {\tt http://www.math.ubc.ca/$\sim$cass/research.html}


\bibitem[Cass2]{Casselman-smooth}
W.~Casselman, {\it Canonical extensions of Harish-Chandra modules to representations of $G$}, Can. J. Math {\bf 16} (1989), 385--438.


\bibitem[CW]{collingwood} D.~Collingwood and W.~McGovern, {\it Nilpotent Orbits in Semisimple Lie Algebras}, Van Nostrand Reinhold, New York 1993.

\bibitem[D]{davenport} H. Davenport, Harold, {\it Multiplicative Number Theory} (3rd ed.), Graduate Texts in Mathematics {\bf 74}, Springer-Verlag, 2000.



\bibitem[DHO]{opdam}  M.~De Martino, V.~Heiermann, and E.~Opdam, {\it On the unramified spherical automorphic spectrum}, arxiv:1512.08566.

\bibitem[SGA3]{SGA3}
M. Demazure, A. Grothendieck, {\it Sch\'emas en groupes III}, Lecture Notes
in Math {\bf 153}, Springer-Verlag, New York, 1970.


\bibitem[GMV]{GMV} M.B.~Green, S.D.~Miller, and P.~Vanhove, {\it Small representations, string instantons,
and Fourier modes Of Eisenstein series, with appendix ``Special unipotent representations'' by Dan Ciubotaru and Peter E. Trapa},
Journal of Number Theory {\bf 146} (2015), pp. 187--309.


\bibitem[I-M]{Iwahori-Matsumoto}
N. Iwahori, N., H. Matsumoto, {\it On some Bruhat decomposition and the structure of the Hecke rings of $p$-adic Chevalley groups}, Inst. Hautes \'Etudes Sci. Publ. Math. No. {\bf 25} (1965), 5--48.



\bibitem[K1]{Kim-G2}
H.H.~Kim, {\it The residual spectrum of $G_2$},
   Can. J. Math. {\bf 48} (1996), pp. 1245-1272.



\bibitem[K2]{Kim-borel} H. Kim, {\it Residual spectrum of split classical groups: contribution from Borel subgroups}, Pacific J. Math. {\bf 199} (2001), no. 2, 417--445.

\bibitem[KS]{Kim-Shahidi} H. Kim and F. Shahidi, {\it Quadratic unipotent Arthur parameters and residual spectrum of symplectic groups}, Amer. J. Math. {\bf 118} (1996), 401--425.

\bibitem[Kn]{Knapp} A.W.~Knapp, {\it Representation theory of semisimple groups.
An overview based on examples}. Princeton Mathematical Series, {\bf 36}. Princeton University Press, Princeton, NJ, 1986.


\bibitem[L]{L} R.P.~Langlands, {\it On the functional equations satisfied by Eisenstein series}, Springer Verlag Lecture Notes in Mathematics, {\bf 544} (1976).

\bibitem[Mi]{M} S.D. Miller,  {\it Residual automorphic forms and spherical unitary representations of exceptional groups}, Ann. of Math. (2) {\bf 177} (2013), no. 3, 1169--1179.

\bibitem[Mo]{Moeglin} C. M{\oe}glin, {\it Repr\'esentations unipotentes et formes automorphes de carr\'e
   int\'egrable}, Forum Math. {\bf 6} (1994), pp. 651--744.

\bibitem[M-W]{MW} C. M{\oe}glin and J.-L. Waldspurger, {\it Spectral decomposition and Eisenstein series. Une paraphrase de l'\'Ecriture}, Cambridge Tracts in Mathematics,  {\bf 113}. Cambridge University Press, Cambridge, 1995.

\bibitem[N]{Ngo} B.C.~Ngo, {\it Le lemme fondamental pour les algèbres de Lie}, Publications Mathématiques de l'IH\'ES {\bf 111}  (2010), pp. 1--169.


\bibitem[R]{Reeder} M. Reeder, {\it Torsion automorphisms of simple Lie algebras}, L'Enseignement Math\'ematique {\bf 56} (2010), 3--47.

\bibitem[Sch]{Schiffmann}
G. Schiffmann,
 {\it
Int\'egrales d'entrelacement et fonctions de Whittaker}, Bull. Soc. Math. France {\bf 99} (1971), 3--72.

\bibitem[Sh1]{Shahidi} F. Shahidi, {\it Whittaker models for real groups}, Duke Math. J., {\bf 47} (1980), 99-125.

\bibitem[Sh2]{Shahidi81} F. Shahidi, {\it On certain $L$-functions}, Amer. J. Math., {\bf 103} (1981), 297--355.

\bibitem[Sh3]{Shahidi-local} F. Shahidi, {\it Local
coefficients and normalization of intertwining operators for $GL(n$)}, Comp. Math. {\bf 48}
(1983), 271--295.

\bibitem[Si]{Silberger}  A. Silberger, {\it Introduction to harmonic analysis on reductive $p$-adic groups. Based on lectures by Harish-Chandra at the Institute for Advanced Study, 1971--1973.} Mathematical Notes, {\bf 23}. Princeton University Press, Princeton, N.J.; University of Tokyo Press, Tokyo, 1979.

\bibitem[Spr]{Springer} T.A. Springer, {\it Linear Algebraic Groups, 2nd Ed.,}
Birkh\"auser, 1998.




\bibitem[St]{Steinberg}  R. Steinberg, {\it Lectures on Chevalley Groups},
Yale Univ. Lecture Notes, 1967.

\bibitem[V]{Vogan} D. Vogan, {\it The unitary dual of $G_2$}, Invent. math. {\bf 116} (1994), 667-791.

\bibitem[VW]{vogan-wallach} D. Vogan and N. Wallach, {\it Intertwining operators for real reductive groups}, Adv. in Math. {\bf 82} (1990), 203--243.

\bibitem[Wa]{Wallach} N. Wallach, {\it Real Reductive Groups II}, Academic Press,  Pure and Applied Mathematics, vol. {\bf 132-II}, 1992.

\bibitem[Wi]{Winarsky} N. Winarsky, {\it Reducibility of principal series representations of $p$-adic Chevalley groups}, Amer. J. Math.  {\bf 100} (1978), 941--956.

\bibitem[Z]{Yuanli} Y. Zhang, {\it The holomorphy and nonvanishing of
normalized
local intertwining operators}, Pacific J. of Math. {\bf 180} (1997), 385--398.

\end{thebibliography}
\end{document}